\documentclass[a4paper]{amsart}

\usepackage{a4wide}
\usepackage[colorlinks, final]{hyperref}
\usepackage[utf8]{inputenc}

\usepackage{amsmath, amscd, amsfonts, amsthm, amssymb, bbm, mathtools}
\usepackage{aliascnt}
\usepackage{comment}
\usepackage{xkvltxp, fixme}

\fxsetup{targetlayout=color}

\makeatletter

\newcommand{\@pmath}{$p$}

\title{The probability that a \texorpdfstring{\@pmath}{p}-adic random \'etale algebra is an unramified field}
\author{Roy Shmueli}

\address{Raymond and Beverly Sackler School of Mathematical Sciences, Tel Aviv University, Tel Aviv 69978, Israel}
\email{royshmueli@mail.tau.ac.il}

\newcommand{\autotheorem}[3]{
\newaliascnt{#1counter}{#2}
\newtheorem{#1}[#1counter]{#3}
\expandafter\newcommand\csname #1counterautorefname\endcsname{#3}
}

\theoremstyle{plain}
\newtheorem{theorem}{Theorem}

\autotheorem{thm}{theorem}{Theorem}
\autotheorem{lem}{theorem}{Lemma}
\autotheorem{prop}{theorem}{Proposition}
\autotheorem{cor}{theorem}{Corollary}
\autotheorem{fact}{theorem}{Fact}
\autotheorem{conj}{theorem}{Conjecture}

\theoremstyle{definition}
\autotheorem{defn}{theorem}{Definition}
\autotheorem{exm}{theorem}{Example}

\theoremstyle{remark}
\autotheorem{rem}{theorem}{Remark}

\newtheoremstyle{case}%
  {\topsep}%
  {\topsep}%
  {}%
  {\parindent}%
  {\itshape}%
  {}%
  { }%
  {\thmname{#1}\thmnumber{ #2}:{\thmnote{ #3.}}}

\theoremstyle{case}
\newtheorem{case}{Case}
\counterwithin*{case}{theorem}

\hypersetup{
  pdftitle={\@title},
  pdfauthor={\@author}
}

\newcommand*{\NN}{{\mathbb{N}}}
\newcommand*{\ZZ}{{\mathbb{Z}}}
\newcommand*{\QQ}{{\mathbb{Q}}}
\newcommand*{\RR}{{\mathbb{R}}}
\newcommand*{\CC}{{\mathbb{C}}}
\newcommand*{\FF}{{\mathbb{F}}}

\renewcommand*{\Pr}{{\mathbb{P}}}
\newcommand*{\Ex}{{\mathbb{E}}}
\newcommand*{\cond}{\;\middle|\;}

\newcommand*{\diff}{\,{\!\mathop{}\mathrm{d}}}

\DeclarePairedDelimiter{\pa}{\lparen}{\rparen}
\DeclarePairedDelimiter{\br}{\lbrack}{\rbrack}
\DeclarePairedDelimiter{\tb}{<}{>}
\DeclarePairedDelimiter{\set}{\{}{\}}
\DeclarePairedDelimiter{\abs}{\lvert}{\rvert}

\newcommand*{\Bcl}{\mathcal{B}}
\newcommand*{\Ccl}{\mathcal{C}}

\newcommand*{\Mcl}{\mathcal{M}}
\newcommand*{\Ocl}{\mathcal{O}}

\newcommand*{\Scl}{\mathcal{S}}

\newcommand*{\rv}{\xi}
\newcommand*{\plift}{\mathfrak{p}}

\newcommand*{\@newrootsAll}[1][K]{R_{#1}}
\newcommand*{\@newrootsSet}[2][K]{{R_{#1}\pa*{#2}}}
\newcommand*{\newroots}{\@ifstar{\@newrootsAll}{\@newrootsSet}}

\newcommand*{\@dfuncmultivar}{J}
\newcommand*{\@dfuncstar}{J^\ast}
\newcommand*{\dfunc}{\@ifstar{\@dfuncstar}{\@dfuncmultivar}}

\newcommand*{\iae}{\varepsilon}

\newcommand*{\pav}[2][p]{\abs*{#2}_{#1}}
\newcommand*{\maxideal}[1][K]{\mathfrak{m}_{#1}}
\newcommand*{\xdisc}[2][K/F]{\Delta_{#1}\pa*{#2}}
\newcommand*{\fdisc}[1][K]{D_{#1}}
\newcommand*{\anorm}[2][K/\QQ_p]{N_{#1}\pa*{#2}}
\newcommand*{\gen}{\Theta}

\newcommand*{\inidicator}{\mathbbm{1}}

\DeclareMathOperator{\id}{id}

\DeclareMathOperator{\Aut}{Aut}
\DeclareMathOperator{\Gal}{Gal}
\DeclareMathOperator{\Hom}{Hom}
\DeclareMathOperator{\Span}{Span}
\DeclareMathOperator{\Inv}{Inv}

\makeatother

\begin{document}

\begin{abstract}
% !TEX root = ./root.tex
We study the random \'etale algebra generated by a random polynomial with i.i.d.\ coefficients distributed according to Haar measure normalized on $\ZZ_p$.
We determine the probability that this random algebra is an unramified field, explicitly.
In addition, we prove a private case of a conjecture made by Bhargava, Cremona, Fisher and Gajovi\'c.
More precisely, we show that this probability is rational function of $p$ that is invariant under replacing $p$ by $1/p$.

\end{abstract}

\maketitle

\listoffixmes

% !TEX root = ./root.tex
\section{Introduction}
\label{sec:introduction}
For any positive integer $n$, let $f_n$ be the random polynomial
\begin{equation*}
  f_n\pa*{X} = \rv_0 + \rv_1 X + \dots + \rv_n X^n,
\end{equation*}
where $\rv_0,\dots,\rv_n$ are independent and identically distributed random variables taking values in $\ZZ_p$ and distributed according to the normalized Haar measure on $\ZZ_p$.
In this paper, we study the random algebra $A_n := \QQ_p\br*{X} / \tb*{f_n}$.
This algebra is \'etale, almost surely.
Therefore, it induces a splitting type.

A splitting type of degree $n$ is a tuple $\sigma = \pa*{d_1^{e_1} \dots d_k^{e_k}}$, where $d_i$ and $e_i$ are positive integers satisfying $\sum_{i=1}^k d_i e_i = n$.
We allow repeats in the list of symbols $d_i^{e_i}$, but the order in which they appear does not matter.

For an \'etale algebra extension $A/\QQ_p$ of degree $n$, we define its splitting type to be $\sigma\pa*{A} = \pa*{d_1^{e_1} \dots d_k^{e_k}}$ if $p$ factors in $A$ as $\plift_1^{e_1} \cdots \plift_k^{e_k}$ where $\plift_1, \dots, \plift_k$ are primes in $A$ having residue field degrees $d_1, \dots, d_k$ respectively.

In \cite{bhargava2022density}, Bhargava, Cremona, Fisher and Gajovi\'c study the splitting type of $A_n$ via the number of roots of $f_n$ in $\QQ_p$.
They calculated the probability that $f_n$ has exactly $r$ roots.
This event is equivalent to the event that $1^1$ appear exactly $r$ times in $\sigma\pa*{A_n}$.
In their research they showed that this probability is a rational function in $p$ which is invariant under replacing $p$ by $p^{-1}$.
Moreover, they conjectured a more general property on the probabilities of the splitting type:

For a splitting type $\sigma$ of degree $n$, let $E_\sigma$ be the event that $A_n$ is \'etale over $\QQ_p$ and $\sigma\pa*{A_n} = \sigma$.
We define the following probabilities:
\begin{align*}
\rho\pa*{\sigma; p} &= \Pr\pa*{E_\sigma}, \\
\alpha\pa*{\sigma; p} &= \Pr\pa*{E_\sigma \cond f_n \text{ monic}}, \qquad \text{and} \\
\beta\pa*{\sigma; p} &= \Pr\pa*{E_\sigma \cond f_n \text{ monic and } f_n \equiv X^n \pmod{p}}.
\end{align*}
\begin{conj}
\label{main-conjecture}
Let $\sigma$ be any splitting type.
Then $\rho\pa*{\sigma; p}$, $\alpha\pa*{\sigma; p}$ and $\beta\pa*{\sigma; p}$ are rational functions of $p$ and satisfy the identities:
\begin{align}
\label{eq:rho-inversion}
\rho\pa*{\sigma; p} &= \rho\pa*{\sigma; p^{-1}}, \qquad \text{and} \\
\label{eq:alphabeta-inversion}
\alpha\pa*{\sigma; p} &= \beta\pa*{\sigma; p^{-1}}.
\end{align}
\end{conj}

We establish the conjecture in the case that the algebra is an unramified field, that is when $\sigma = \pa*{n^1}$.
\begin{theorem}
\label{main-thm}
For any positive integer $n$ there exists a rational function $\dfunc*_n \in \QQ\pa*{t}$ such that
for $\sigma = \pa*{n^1}$ we have
\begin{align}
\label{eq:rho-formula}
\rho\pa*{\sigma; p} &= \frac{p-1}{p^{n+1} - 1} \pa*{p^n \dfunc*_n\pa*{p} + \dfunc*_n\pa*{p^{-1}}}, \\
\label{eq:alpha-formula}
\alpha\pa*{\sigma; p} &= \dfunc*_n\pa*{p}, \qquad \text{and}\\
\label{eq:beta-formula}
\beta\pa*{\sigma; p} &= \dfunc*_n\pa*{p^{-1}}.
\end{align}
\end{theorem}

The function $\dfunc*_n$ in \autoref{main-thm} is given by explicit recursive formula:
Let $\dfunc_1\pa*{u, v} = 1$ and for $n > 1$ let
\begin{equation}
\label{eq:d-function-rec}
\dfunc_n \pa*{u, v} = \frac{1}{u^{n-1} - v^{n - 1}} \sum_{1\ne d \mid n} v^{n / d - 1} \pa*{\sum_{e\mid d} \mu\pa*{\frac{d}{e}}u^{e - 1}} \dfunc_{n/d} \pa*{u^d, v}.
\end{equation}
Here $\mu$ is the M\"obius function.
We then define $\dfunc*_n$ by
\begin{equation}
\label{eq:d-star-function}
\dfunc*_n \pa*{t} = \frac{1}{n} \dfunc_n\pa*{t, t^{-n/2}}.
\end{equation}
When $n$ is odd, the powers of $v$ on the right side of \eqref{eq:d-function-rec} are even except in the arguments of $\dfunc_{n/d}$, therefore an inductive argument gives that $\dfunc*_n$ is rational.

% !TEX root = ./root.tex

\subsection*{Acknowledgements}
I thank Itai Bar-Deroma and Sahar Diskin for their feedback on the research.
I also thank Eli Glasner for his support in the research and my supervisor, Lior Bary-Soroker for his conversions and ideas.

This research was partially supported by grants from the Israel Science Foundation, grant no.\ 702/19 and grant no.\ 1194/19.

% !TEX root = ./root.tex

\section{Notations and generalities}

Let $F$ be a $p$-adic field.
We denote its valuation ring by $\Ocl_F$ and its maximal ideal by $\maxideal[F]$.

\subsection{Absolute value}
We denote by $\pav[F]{\cdot}$ the $p$-adic absolute value normalized such that $\pav[F]{\pi}=q^{-1}$, where $\pi$ is the uniformizer of $F$ and $q$ is the size of the residue field of $F$.
In the case $F = \QQ_p$, we abbreviate and write $\pav{\cdot}$ instead of $\pav[\QQ_p]{\cdot}$.

Each absolute value $\pav[F]{\cdot}$ can be extended uniquely to an absolute value on the algebraic closure of $\QQ_p$, and all of those extensions are equivalent.
For field extension $K/F$ of degree $n$, we have that $\pav[K]{x} = \pav[F]{x}^n$ for any $x$ in the algebraic closure of $\QQ_p$.

\subsection{Haar measure}
We denote with $\lambda_F$ the normalized Haar measure on $F$, that is the unique measure which is locally finite, regular, invariant to translations and satisfy $\lambda_F\pa*{\Ocl_F} = 1$.
We have that for any $\alpha \in F$ and a Borel set $S\subseteq F$ then $\lambda_F\pa*{\alpha S} = \pav[F]{\alpha} \lambda_F\pa*{S}$.

In this paper, all the integrals on $F$ are Lebesgue integrals according to $\lambda_F$.

\subsection{Discriminant}

Let $K/F$ be a field extension of degree $n$.
For elements $\omega_1, \dots, \omega_n \in K$ we define their discriminant to be
\begin{equation*}
  \xdisc{\omega_1, \dots, \omega_n} = \det\pa*{\pa*{T_{K/F}\pa*{\omega_i \omega_j}}_{\substack{1\le i \le n \\ 1 \le j \le n}}},
\end{equation*}
where $T_{K/F}$ is the field trace.
If $L$ is $F$-linear map from $K$ to $K$ and mapping $\omega_1, \dots, \omega_n$ to $\omega'_1, \dots, \omega'_n$ then
\begin{equation}
  \label{eq:g:disc-conversion}
  \xdisc{\omega'_1, \dots, \omega'_n} = \pa*{\det L}^2 \xdisc{\omega_1, \dots, \omega_n}.
\end{equation}

In addition, for $x\in K$, we define the discriminant of $x$ to be $\xdisc{x} = \xdisc{1, x, \dots, x^{n-1}}$.
We have that $K = F\br*{x}$ if and only if $\xdisc{x} \ne 0$.
Moreover, if $K = F\br*{x}$ then $\xdisc{x}$ is the discriminant of the minimal polynomial over $F$.

When $\omega_1, \dots, \omega_n$ are $\Ocl_F$-basis of $\Ocl_K$ then $\xdisc{\omega_1, \dots, \omega_n}$ is the field discriminant of $K$ relative of $F$ which we denote with $\fdisc[K/F]$.
The field discriminant is not depends on the choice of $\omega_1, \dots, \omega_n$ up to multiplication by a unit.
Therefore, the absolute value of the discriminant, $\pav[F]{\fdisc[K/F]}$, is invariant to the choice of $\omega_1, \dots, \omega_n$.
Also, we have that $K/F$ is unramified if and only if the uniformizer of $F$ divides $\fdisc[K/F]$.

% !TEX root = ./root.tex

\section{The inversion formula}
\label{inc-alg}

In this section we prove an inversion formula for $\dfunc_n$:
\begin{prop}
\label{d-function-inverse}
For any $n \in \NN$,
\begin{equation*}
\dfunc_n\pa*{u^{-1}, v^{-1}} = v^{n-1} \dfunc_n\pa*{u, v}.
\end{equation*}
\end{prop}

We prove this proposition using the theory of incidence algebras which we introduce below.
For more information on the subject see \cite{spiegel1997incidence}.

\subsection{Incidence algebras}
Let $\pa*{P, \le}$ be a poset.
For $x, y \in P$, we define the interval between $x$ and $y$ to be the subset
\begin{equation*}
\br*{x, y}_P = \set*{z \in P : x\le z \le y}.
\end{equation*}
We say that a poset $P$ is a \emph{locally finite poset} if $\br*{x,y}_P$ is finite set for each $x, y \in P$.

A sequence $\underline x = \pa*{x_0, \dots, x_k}$ of elements in $P$ is called a \emph{proper chain in $P$ of length $k + 1$} if
\begin{equation*}
x_0 < x_1 < \dots < x_k.
\end{equation*}
We also say that $\underline x$ starts at $x_0$ and ends at $x_k$.
In addition, let $\Ccl_P^k\pa*{x, y}$ denote the set of all proper chains of length $k + 1$ that start at $x$ and ends at $y$.
And, let $\Ccl_P^\ast\pa*{x, y}$ denotes the set of all proper chains that start at $x$ and ends at $y$, i.e. $\Ccl_P^\ast\pa*{x, y} = \bigcup_{k=0}^\infty \Ccl_P^k\pa*{x, y}$.

For two chains $\underline z, \underline w \in \Ccl_P^\ast\pa*{x, y}$ we say that \emph{$\underline w$ is finer than $\underline z$} if $\underline z$ is a sub-sequence of $\underline w$.
For a proper chain $\underline z = \pa*{z_0, \dots, z_k}$, we denote with $\Ccl_P^m\pa*{z_0, \dots, z_k}$ the set of all proper chains of length $m+1$ which are finer than $\bar z$.
Note that for any non-negative integer $m$ we have
\begin{equation}
\label{eq:ia:refinments-num}
\# \Ccl_P^m \pa*{z_0, \dots, z_k} = \sum_{m_1 + \dots + m_k = m} \; \prod_{i=1}^k \#\Ccl_P^{m_i} \pa*{z_{i-1}, z_i}.
\end{equation}

Let $P$ be a locally finite poset and let $A$ be a commutative ring with unity.
\emph{The incidence $A$-algebra of $P$}, denoted by $I_A\pa*{P}$, is the algebra whose elements are $\iae\colon P^2 \to A$ such that $\iae\pa*{x, y} = 0$ for all $x \not\le y$.
The operations are defined as followed:
\begin{align*}
\pa*{\iae_1 + \iae_2}\pa*{x, y} &= \iae_1\pa*{x,y} + \iae_2\pa*{x, y}, \\
\pa*{\iae_1 \ast \iae_2}\pa*{x, y} &= \sum_{x\le z \le y} \iae_1\pa*{x,z}\iae_2\pa*{z, y}, \quad\text{and} \\
\pa*{a \cdot \iae}\pa*{x,y} &= a \iae\pa*{x,y},
\end{align*}
for all $x,y \in P$, $\iae, \iae_1, \iae_2 \in I_A\pa*{P}$ and $a \in A$.

Let
\begin{equation*}
\delta\pa*{x, y} = \begin{cases}
1, & x = y, \\
0, & \text{otherwise}.
\end{cases}
\end{equation*}
Then $\delta$ is the identity element of $I_A\pa*{P}$.

We have the following lemmas describe the invertible elements in $I_A\pa*{P}$.
\begin{lem}
\label{ia:inverse}
An element $\iae \in I_A\pa*{P}$ is invertible if and only if $\iae\pa*{x, x}$ is invertible in $A$ for each $x\in P$.
In that case, $\iae^{-1}$ is defined recursively: for any $x \in P$, we have that $\iae^{-1}\pa*{x, x} = \iae\pa*{x, x}^{-1}$ and for any $x < y$
\begin{equation*}
\iae^{-1} \pa*{x, y} = - \iae\pa*{x, x}^{-1} \sum_{x < z \le y} \iae\pa*{x, z} \iae^{-1}\pa*{z,y}.
\end{equation*}
\end{lem}
\begin{proof}
See \cite[Theroem 1.2.3]{spiegel1997incidence} and its proof.
\end{proof}

\begin{lem}
\label{ia:chain-inverse}
Assume that $\iae\pa*{x, x} = 1$ for all $x \in P$. Then, for all $x, y \in P$,
\begin{equation*}
\iae^{-1} \pa*{x, y} = \sum_{\pa*{z_0, \dots, z_k} \in \Ccl_P^\ast\pa*{x, y}} \pa*{-1}^k \prod_{i=0}^{k-1} \iae\pa*{z_i, z_{i+1}}.
\end{equation*}
\end{lem}
\begin{proof}
We prove this with induction on $x$.
For the induction base $x = y$, and by \autoref{ia:inverse} we get that $\iae^{-1}\pa*{x,x} = 1$.
Moreover, the sum on the right side of the equation is also $1$ since there is a single chain from $x$ to $x$.

For the induction step, let $x < y$.
By \autoref{ia:inverse} we have that
\begin{equation*}
  \iae^{-1}\pa*{x, y} = - \sum_{x < z \le y} \iae\pa*{x, z} \iae^{-1}\pa*{z,y}.
\end{equation*}
We use the induction assumption to obtain
\begin{equation*}
  \iae^{-1}\pa*{x, y} = - \sum_{x < z \le y} \iae\pa*{x, z} \sum_{\pa*{z_0, \dots, z_k} \in \Ccl_P^\ast\pa*{z, y}} \pa*{-1}^k \prod_{i=0}^{k-1} \iae\pa*{z_i, z_{i+1}}.
\end{equation*}
We rearrange the factors
\begin{equation*}
  \iae^{-1}\pa*{x, y} = \sum_{x < z \le y}  \sum_{\pa*{z_0, \dots, z_k} \in \Ccl_P^\ast\pa*{z, y}} \pa*{-1}^{k+1} \iae\pa*{x, z} \prod_{i=0}^{k-1} \iae\pa*{z_i, z_{i+1}}.
\end{equation*}
Running over all $x < z \le y$ and then over all chains $\pa*{z_0, \dots, z_k} \in \Ccl_P^\ast\pa*{z, y}$ is the same as running over all chains $\pa*{w_0, \dots, w_k} \in \Ccl_P^\ast\pa*{x, y}$ with $z$ being the first element in the chain after $x$.
Thus,
\begin{equation*}
  \iae^{-1}\pa*{x, y} = \sum_{\pa*{w_0, \dots, w_k} \in \Ccl_P^\ast\pa*{z, y}} \pa*{-1}^{k} \prod_{i=0}^{k-1} \iae\pa*{w_i, w_{i+1}}. \qedhere
\end{equation*}
\end{proof}

Let $Q \subseteq P$, by abuse of notation we let $Q$ denotes the poset composed of the elements of $Q$ and the order $\le$ restricted to $Q^2$.
For $\iae \in I_A\pa*{P}$ we have that $\iae|_Q \in I_A\pa*{Q}$.
For $\iae, \iae_1, \iae_2 \in I_A\pa*{P}$, we denote with $\iae_1 \ast_Q \iae_2$ the multiplication of $\iae_1|_Q$ and $\iae_2|_Q$ in $I_A\pa*{Q}$ and with $\Inv_Q\iae$ the inverse of $\iae|_Q$ in $I_A\pa*{Q}$.

\begin{cor}
  \label{ia:equivalence-posets-inverse}
  Let $Q, Q' \subseteq P$ and $x,y \in P$ such that $\br*{x, y}_Q = \br*{x, y}_{Q'}$.
  Then for all $\iae \in I_A\pa*{P}$
  \begin{equation*}
    \Inv_Q \iae\pa*{x, y} = \Inv_{Q'} \iae\pa*{x,y}.
  \end{equation*}
\end{cor}
\begin{proof}
  This is immediate from \autoref{ia:chain-inverse} since $\Ccl_Q^\ast\pa*{x,y} = \Ccl_{Q'}^\ast\pa*{x,y}$.
\end{proof}

We add a definition that will be useful in the next subsections.
\begin{defn}
  Let $x \le y$ and let $Q, Q' \subseteq P$, we say that \emph{$Q$ and $Q'$ are complementing the interval ${\br*{x, y}_P}$} if $x, y \in Q \cap Q'$ and for any $x < z < y$ then either $z \in Q$ or $z \in Q'$ but not both.
\end{defn}

\subsection{The M\"obius function}

We define the following function $\zeta \in I_A\pa*{P}$ by
\begin{equation*}
  \zeta\pa*{x, y} = \begin{cases}
  1, & x \le y, \\
  0, & \text{otherwise}.
  \end{cases}
\end{equation*}
By \autoref{ia:inverse} we have that $\zeta$ is invertible and its inverse, $\mu$, satisfy
\begin{equation}
\label{eq:ia:mobius-rec}
\mu \pa*{x, y} = - \sum_{x < z \le y} \mu\pa*{z, y}.
\end{equation}
Moreover, since $\zeta\pa*{x, x} = 1$ for all $x \in P$, \autoref{ia:chain-inverse} gives that
\begin{equation}
\label{eq:ia:mobius-chains}
\mu \pa*{x, y} = \sum_{k=0}^\infty \pa*{-1}^k \#\Ccl_P^k\pa*{x, y}.
\end{equation}
We call $\mu$ the \emph{M\"obius function on $P$}.
For $Q \subseteq P$, we denote by $\mu_Q$ the M\"obius function on $Q$, i.e. $\mu_Q = \Inv_Q \zeta$.

\begin{lem}
\label{ia:moebius-completition}
Let $x < y$ and let $Q, Q' \subseteq P$ which are complementing the interval ${\br*{x, y}_P}$.
Then
\begin{equation*}
\Inv_Q\mu\pa*{x, y} = - \mu_{Q'}\pa*{x, y}.
\end{equation*}
\end{lem}
\begin{proof}
By \autoref{ia:chain-inverse} we have
\begin{equation*}
\Inv_Q\mu\pa*{x, y} = \sum_{\pa*{z_0, \dots, z_k} \in \Ccl_Q^\ast\pa*{x, y}} \pa*{-1}^k \prod_{i=1}^k \mu\pa*{z_{i-1}, z_i}.
\end{equation*}
By \eqref{eq:ia:mobius-chains} we get
\begin{equation*}
\Inv_Q\mu\pa*{x, y} = \sum_{\pa*{z_0, \dots, z_k} \in \Ccl_Q^\ast\pa*{x, y}} \pa*{-1}^k \prod_{i=1}^k \pa*{\sum_{m=0}^\infty \pa*{-1}^m \#\Ccl_P^m\pa*{z_{i-1}, z_i}}.
\end{equation*}
We expand the multiplication and then use \eqref{eq:ia:refinments-num} to obtain
\begin{align*}
\Inv_Q\mu\pa*{x, y}
  &= \sum_{\pa*{z_0, \dots, z_k} \in \Ccl_Q^\ast\pa*{x, y}} \pa*{-1}^k \sum_{m=0}^\infty \pa*{-1}^m \sum_{m_1 + \dots + m_k = m} \; \prod_{i=1}^k \#\Ccl_P^{m_i}\pa*{z_{i-1}, z_i} \\
  &= \sum_{\pa*{z_0, \dots, z_k} \in \Ccl_Q^\ast\pa*{x, y}} \pa*{-1}^k \sum_{m=0}^\infty \pa*{-1}^m  \#\Ccl_P^m\pa*{z_0, \dots, z_k}.
\end{align*}
Changing the order of summation gives that
\begin{equation}
\label{eq:ia:mc:inv-q-sum}
\Inv_Q\mu\pa*{x, y} = \sum_{m=0}^\infty \pa*{-1}^m \sum_{\pa*{z_0, \dots, z_k} \in \Ccl_Q^\ast\pa*{x, y}} \pa*{-1}^k \#\Ccl_P^m\pa*{\bar z}.
\end{equation}

From the inclusion–exclusion principle we have that
\begin{align*}
\#\Ccl_{Q'}^m\pa*{x, y}
  &= \#\Ccl_P^m\pa*{x, y} - \sum_{\substack{x < z_1 < y \\ z_1 \in Q}} \#\Ccl_P^m\pa*{x, z_1, y} + \sum_{\substack{x < z_1 < z_2 < y \\ z_1, z_2 \in Q}} \#\Ccl_P^m\pa*{x, z_1, z_2, y} - \dots \\
  &= - \sum_{\pa*{z_0, \dots, z_k} \in \Ccl_Q^\ast\pa*{x, y}} \pa*{-1}^k \#\Ccl_P^m\pa*{z_0, \dots, z_k}.
\end{align*}
We plug this into \eqref{eq:ia:mc:inv-q-sum} and get that
\begin{equation*}
\Inv_Q\mu\pa*{x, y} = - \sum_{m=0}^\infty \pa*{-1}^m \#\Ccl_{Q'}^m\pa*{x, y}.
\end{equation*}
Applying \autoref{ia:chain-inverse} on left side of the equation finishes the proof.
\end{proof}

We define the function $\Gamma_Q : P \times Q \to A$ by
\begin{equation}
  \label{eq:ia:gamma-def}
\Gamma_Q\pa*{x, y} = \sum_{\substack{x \le z \le y \\ z \in Q}}\mu\pa*{x, z} \Inv_Q \mu\pa*{z, y}.
\end{equation}

\begin{lem}
\label{ia:gamma-formula}
Let $Q\subseteq P$. Then for all $x \in P$ and $y \in Q$ we have that
\begin{equation*}
\Gamma_Q\pa*{x, y} = \begin{cases}
- \Inv_{Q\cup \set*{x}}\mu \pa*{x, y}, &x \notin Q, \\
1, & x = y, \\
0, & \text{otherwise}.
\end{cases}
\end{equation*}
\end{lem}
\begin{proof}
We start with the case of $x \in Q$.
In this case we can write $\Gamma_Q\pa*{x, y}$ as multiplication in $I_A\pa*{Q}$ as follows
\begin{equation*}
\Gamma_Q\pa*{x, y}
  = \pa*{\mu \ast_Q \Inv_Q\mu}\pa*{x, y}.
\end{equation*}
Clearly, $\mu\ast_Q \Inv_Q\mu = \delta$ and this finish the case when $x\in Q$.

If $x \notin Q$, then we have by \autoref{ia:inverse} that
\begin{equation*}
- \Inv_{Q\cup \set*{x}}\mu \pa*{x, y}
  = \sum_{\substack{x < z \le y \\ z \in Q\cup \set*{x}}} \mu\pa*{x, z} \Inv_{Q\cup \set*{x}}\mu \pa*{z, y}.
\end{equation*}
From \autoref{ia:equivalence-posets-inverse} we conclude that $\Inv_{Q\cup\set*{x}} \mu\pa*{z, y} = \Inv_Q \mu\pa*{z, y}$ for all $x < z$.
Hence,
\begin{equation*}
- \mu_{Q\cup \set*{y}} \pa*{x, y}
  = \sum_{\substack{x \le z \le y \\ z \in Q}} \mu\pa*{x, z} \Inv_{Q}\mu \pa*{z, y} \\
  = \Gamma_Q\pa*{x, y}. \qedhere
\end{equation*}
\end{proof}

\subsection{The \texorpdfstring{$\theta$}{theta} polynomial}

Set $A_P = \QQ\br*{t^\pm_x : x \in P}$ where $t_x$ are algebraically independent elements over $\QQ$.
We define the following function $\theta \in I_{A_P}\pa*{P}$ by
\begin{equation}
\label{eq:theta-poly}
\theta\pa*{x, y; \underline{t}} = \sum_{x\le z \le y} \mu \pa*{z, y} \frac{t_z}{t_x}
\end{equation}
If there is no risk for ambiguity we write $\theta\pa*{x, y}$ instead of $\theta\pa*{x, y; \underline{t}}$.
Note that $\theta\pa*{x, x} = 1$ for all $x \in P$, hence $\theta$ is invertible by \autoref{ia:inverse}.
We are interested in finding $\Inv_Q \theta$ for different subposets $Q \subseteq P$.

\begin{defn}
  Let $Q\subseteq P$.
  We say that a monomial $m \in A_P$ is \emph{$Q$-admissible} if there exists
  $z_1, \dots, z_k, y \in Q$ and $w_1, \dots, w_k \in P$ such that
  \begin{gather*}
    m = \frac{t_{w_1}}{t_{z_1}} \cdots \frac{t_{w_k}}{t_{z_k}}, \qquad \text{and} \\
    z_1 \le w_1 \le z_2 \le w_2 \le \dots \le z_k \le w_k \le y.
  \end{gather*}
  If $\frac{t_{w_1}}{t_{z_1}} \cdots \frac{t_{w_k}}{t_{z_k}}$ is a reduced fraction that is, $w_i \ne z_j$ for all $i$ and $j$, we call it \emph{the admissible form of $m$}.
\end{defn}
We denote by $\Mcl\pa*{Q}$ the set of all $Q$-admissible monomials.

\begin{lem}
\label{tp:theta-inv-coeffs}
Let $Q\subseteq P$ and let $x, y \in Q$. Then $\Inv_Q\theta\pa*{x, y} \in \Span_\QQ \Mcl\pa*{\br*{x, y}_Q}$.
Moreover, let $m \in \Mcl\pa*{\br*{x, y}_Q}$ with admissible form $\frac{t_{w_1}}{t_{z_1}}\cdots \frac{t_{w_k}}{t_{z_k}}$ then the coefficient of $m$ in $\Inv_Q\theta\pa*{x, y}$ is:
\begin{equation}
  \label{eq:tp:theta-inv-coeffs}
\br*{\Inv_Q\theta\pa*{x, y}}_m = \Inv_Q \mu\pa*{x, z_1} \prod_{i=1}^k \mu_{Q_i} \pa*{z_i, w_i} \, \Gamma_Q\pa*{w_i, z_{i+1}} ,
\end{equation}
where $z_{k+1} = y$ and $Q_i = Q \cup \set*{w_i}$.
\end{lem}
\begin{proof}
We prove this with induction on $x$.
For the induction base we have that $x = y$.
In this case, $\Mcl\pa*{\br*{x, y}_Q} = \set*{1}$ and $\Inv_Q\theta\pa*{x, y} = 1$ so the lemma is clear.

For the induction step we have $y > x$, and we assume that for all $y \ge z > x$ we have $\Inv_Q\theta\pa*{z, y} \in \Span_\QQ \Mcl\pa*{\br*{z, y}_Q}$ and the coefficients of $\Inv_Q \theta\pa*{z, y}$ are as described in this lemma.

By \autoref{ia:inverse},
\begin{equation}
\label{eq:tp:tic:inverse-formula}
\Inv_Q\theta\pa*{x, y} = - \sum_{\substack{x < z \le y \\ z \in Q}} \theta\pa*{x, z} \cdot \Inv_Q  \theta\pa*{z, y}.
\end{equation}
From definition, we have that $\theta\pa*{x, z} \in \Span_\QQ \set*{t_w / t_x : x \le w \le z}$.
Together with the induction assumption we infer that $\theta\pa*{x, z} \cdot \Inv_Q \theta\pa*{z, y} \in \Span_\QQ \Mcl\pa*{\br*{x, y}_Q}$ for all $x < z \le y$.
Therefore, we get from \eqref{eq:tp:tic:inverse-formula} that $\Inv_Q\theta\pa*{x, y} \in \Span_\QQ \Mcl\pa*{\br*{x, y}_Q}$.

Next, we prove \eqref{eq:tp:theta-inv-coeffs}.
Let $m \in \Mcl\pa*{\br*{x, y}_Q}$ with admissible form $\frac{t_{w_1}}{t_{z_1}}\cdots \frac{t_{w_k}}{t_{z_k}}$.
Set $m' = \frac{t_{w_2}}{t_{z_2}}\cdots \frac{t_{w_k}}{t_{z_k}}$ and $C = \br*{\Inv_Q \theta\pa*{z_2, y}}_{m'}$.
So by the induction's assumption and since $\mu_Q\pa*{z_2, z_2} = 1$ we have that
\begin{equation}
\label{eq:tp:tic:c-formula}
C = \prod_{i=2}^{k} \mu_{Q_i} \pa*{z_i, w_i} \, \Gamma_Q\pa*{w_i, z_{i+1}} .
\end{equation}

We split the proof into two cases:
\begin{case}[When $z_1 = x$]
  We take a look at the sum in \eqref{eq:tp:tic:inverse-formula}, and we identify the $z$ which contributes to the coefficient of $m$.
  Since $m \notin \Mcl\pa*{\br*{x, y}_Q}$, there are two kinds of summands that contains $m$:
  \begin{itemize}
    \item Summands with $z_1 < z < w_1$.
      Here $m$ appears from the multiplication of the monomial $\frac{t_z}{t_{z_1}}$ in $\theta\pa*{x, z}$ and the monomial $\frac{t_{w_1}}{t_z} \cdot m'$ in $\Inv_Q\pa*{z, y}$.

    \item Summands with $w_1 \le z \le z_2$.
      Here $m$ appears from the multiplication of the monomial $\frac{t_{w_1}}{t_{z_1}}$ in $\theta\pa*{x, z}$ and the monomial $m'$ in $\Inv_Q\pa*{z, y}$.
  \end{itemize}
  So from \eqref{eq:tp:tic:inverse-formula} we get that
  \begin{multline*}
    \br*{\Inv_Q\theta\pa*{x, y}}_m
    =
    - \sum_{\substack{z_1 < z < w_1 \\ z \in Q}} \br*{\theta\pa*{z_1, z}}_{t_z/t_{z_1}} \, \br*{\Inv_Q \theta\pa*{z, y}}_{t_{w_1}/t_z \cdot m'} \\
    - \sum_{\substack{w_1 \le z \le z_2 \\ z \in Q}} \br*{\theta\pa*{z_1, z}}_{t_{w_1}/t_{z_1}} \, \br*{\Inv_Q \theta\pa*{z, y}}_{m'}.
  \end{multline*}
  And from the definition of $\theta$ (see \eqref{eq:theta-poly}) we have that
  \begin{equation}
    \label{eq:tp:tic:coeff-expansion}
    \br*{\Inv_Q\theta\pa*{x, y}}_m
    =
    - \sum_{\substack{z_1 < z < w_1 \\ z \in Q}} \br*{\Inv_Q \theta\pa*{z, y}}_{t_{w_1}/t_z \cdot m'} \\
    - \sum_{\substack{w_1 \le z \le z_2 \\ z \in Q}} \mu\pa*{w_1, z} \, \br*{\Inv_Q \theta\pa*{z, y}}_{m'}.
  \end{equation}

  Next, we focus on the term $\br*{\Inv_Q \theta\pa*{z, y}}_{t_{w_1}/t_z \cdot m'}$ when $x < z < w_1$.
  Using the induction assumption gives
  \begin{equation*}
    \br*{\Inv_Q \theta\pa*{z, y}}_{t_{w_1}/t_z \cdot m'}
    = \Inv_Q \mu\pa*{z, z} \, \mu_{Q_1} \pa*{z, w_1} \, \Gamma_Q\pa*{w_1, z_2}  {\prod_{i=2}^k  \mu_{Q_i}\pa*{z_i, w_i} \, \Gamma_Q\pa*{w_1, z_{i+1}}} .
  \end{equation*}
  And by \eqref{eq:tp:tic:c-formula} and $\Inv_Q\mu\pa*{z, z} = 1$ we obtain
  \begin{equation}
    \label{eq:tp:tic:first_terms}
    \br*{\Inv_Q \theta\pa*{z ,y}}_{t_{w_1} / t_z \cdot m'} =  \mu_{Q_1} \pa*{z, w_1} \, \Gamma_Q\pa*{w_1, z_2} \cdot C.
  \end{equation}

  We move forward to the term $\br*{\Inv_Q \theta\pa*{z, y}}_{m'}$ when $w_1 \le z \le z_2$.
  From the induction assumption
  \begin{equation*}
    \br*{\Inv_Q \theta\pa*{z, y}}_{m'}
    = \Inv_Q \mu\pa*{z, z_2} {\prod_{i=2}^{k} \mu_{Q_i}\pa*{w_i, z_i}} \,\Gamma_Q\pa*{w_i, z_{i+1}}  .
  \end{equation*}
  And plugging in \eqref{eq:tp:tic:c-formula} gives
  \begin{equation}
    \label{eq:tp:tic:second_terms}
    \br*{\Inv_Q \theta\pa*{z, y}}_{m'} = \Inv_Q  \mu\pa*{z, z_2} \cdot C.
  \end{equation}

  We plug \eqref{eq:tp:tic:first_terms} and \eqref{eq:tp:tic:second_terms} into \eqref{eq:tp:tic:coeff-expansion} to get
  \begin{equation}
    \label{eq:tp:tic:coeff-expansion-2}
    \br*{\Inv_Q\theta\pa*{x, y}}_m
    = - \sum_{\substack{x < z < w_1 \\ z \in Q}} \mu_{Q_1} \pa*{z, w_1} \,\Gamma_Q\pa*{w_1, z_2} \cdot C
    - \sum_{\substack{w_1 \le z \le z_2 \\ z \in Q}} \mu\pa*{w_1, z} \Inv_Q\mu\pa*{z, z_2} \cdot C.
  \end{equation}

  We look on the second sum of \eqref{eq:tp:tic:coeff-expansion-2}.
  So from \eqref{eq:ia:gamma-def} and since $\mu_{Q_1} \pa*{w_1, w_1} = 1$ we get
  \begin{equation*}
    \sum_{\substack{w_1 \le z \le z_2 \\ z \in Q}} \mu\pa*{w_1, z} \Inv_Q\mu\pa*{z, z_2} \cdot C = \mu_{Q_1}\pa*{w_1, w_1} \Gamma_Q\pa*{w_1, z_2} \cdot C.
  \end{equation*}
  We put this in \eqref{eq:tp:tic:coeff-expansion-2} and get that
  \begin{equation*}
    \br*{\Inv_Q\theta\pa*{x, y}}_m = - \sum_{\substack{z_1 < z \le w_1 \\ z \in Q_1}}  \mu_{Q_1} \pa*{z, w_1} \cdot \Gamma_Q\pa*{w_1, z_2} C
  \end{equation*}
  Plugging \eqref{eq:ia:mobius-rec}, \eqref{eq:tp:tic:c-formula} and $\Inv_Q \mu \pa*{x, z_1} = 1$ into the last equation gives
  \begin{align*}
    \br*{\Inv_Q\theta\pa*{x, y}}_m
    &= \mu_{Q_1}\pa*{z_1, w_1} \, \Gamma_Q\pa*{w_1, z_2} C \\
    &= \Inv_Q \mu \pa*{x, z_1} \prod_{i=1}^k \mu_{Q_i}\pa*{z_i, w_i} \,\Gamma_Q\pa*{w_i, z_{i+1}},
  \end{align*}
  as needed.
\end{case}

\begin{case}[When $x < z_1$]
  Since for all $x < z \le y$, all the monomials of $\theta\pa*{x, z}$ contains $t_x$ except for $1$, from \eqref{eq:tp:tic:inverse-formula} we have
  \begin{equation*}
    \br*{\Inv_Q\theta\pa*{x, y}}_m
    = - \sum_{\substack{x < z \le y \\ z \in Q}} \br*{\theta\pa*{x, z}}_1 \, \br*{\Inv_Q \theta\pa*{z, y}}_m.
  \end{equation*}
  Thus, \eqref{eq:theta-poly} implies
  \begin{equation*}
    \br*{\Inv_Q\theta\pa*{x, y}}_m
    = - \sum_{\substack{x < z \le y \\ z \in Q}} \mu\pa*{x, z} \br*{\Inv_Q \theta\pa*{z, y}}_m.
  \end{equation*}
  Then, we plug in the induction assumption to obtain
  \begin{equation*}
    \br*{\Inv_Q\theta\pa*{x, y}}_m
    = - \sum_{\substack{x < z \le y \\ z \in Q}} \mu\pa*{x, z} \Inv_Q \mu\pa*{z, z_1} \prod_{i=1}^k  \mu_{Q_i}\pa*{z_i, w_i} \, \Gamma_Q\pa*{w_i, z_{i+1}},
  \end{equation*}
  and we finish the induction by using \autoref{ia:inverse}.
\end{case}
\end{proof}

\begin{lem}
  \label{tp:theta-inversion-coeff}
  Let $x < y$ and let $Q, Q' \subseteq P$ complementing the interval ${\br*{x, y}_P}$, and let $m$ be a $\br*{x,y}_Q$-admissible monomial.
  Suppose that the monomial $m^{-1} \cdot t_y / t_x$ is $\br*{x, y}_{Q'}$-admissible.
  Then,
  \begin{equation*}
    \br*{\Inv_Q \theta\pa*{x, y}}_m = - \br*{\Inv_{Q'} \theta\pa*{x, y}}_{m^{-1} \cdot t_y / t_x}.
  \end{equation*}
\end{lem}
\begin{proof}
  Let $\frac{t_{w_1}}{t_{z_1}} \cdots \frac{t_{w_k}}{t_{z_k}}$ be the admissible form of $m$.
  Set $z_{k+1} = y$, $Q_i = Q \cup \set*{w_i}$ and $Q'_i = Q' \cup\set*{z_i}$.
  So from \autoref{tp:theta-inv-coeffs},
  \begin{equation}
    \label{eq:tic:m-coeff}
    \br*{\Inv_Q \theta\pa*{x, y}}_m
        = \Inv_Q \mu\pa*{x, z_1} \prod_{i=1}^k \mu_{Q_i}\pa*{z_i, w_i} \Gamma_Q \pa*{w_i, z_{i+1}}.
  \end{equation}
  Next, we have that
  \begin{equation*}
    m^{-1} \cdot \frac{t_y}{t_x} = \frac{t_{z_1}}{t_x} \cdot\frac{t_{z_2}}{t_{w_1}} \cdots \frac{t_{z_k}}{t_{w_{k-1}}} \cdot \frac{t_y}{t_{w_k}}.
  \end{equation*}
  Hence, $w_1, \dots, w_{k-1}$ must be in the denominator of $m^{-1} \cdot t_y / t_x$.
  And since $m^{-1} \cdot t_y / t_x$ is $\br*{x, y}_{Q'}$-admissible we get that $w_1, \dots, w_{k-1} \in Q'$.
  Also, $w_k \in Q'$ since either $w_k = y$ or $w_k$ is in the denominator of $m^{-1} \cdot t_y / t_x$.

  Finally, we split the proof into 4 cases.
  In each of those cases $m^{-1} \cdot t_y/t_x$ has a slightly different admissible form.
  \begin{case}[When $x = z_1$ and $y \ne w_k$]
    \label{tic:case-half-close}
    Since $Q$ and $Q'$ are complementing $\br*{x, y}$ and $x \le z_1 < w_i \le w_k < y$ for all $i = 1, \dots ,k$, the elements $w_i\notin Q$.
    Thus, we use \autoref{ia:gamma-formula} to get
    \begin{equation*}
      \prod_{i=1}^k \mu_{Q_i}\pa*{z_i, w_i} \, \Gamma_Q\pa*{w_i, z_{i+1}}
        = \prod_{i=1}^k \mu_{Q_i}\pa*{z_i, w_i} \, \pa*{-\Inv_{Q_i} \mu \pa*{w_i, z_{i+1}}}.
    \end{equation*}
    The subposets $Q_i$ and $Q'_i$ are complementing the interval $\br*{z_i, w_i}$.
    Moreover, $Q_i$ and $Q'_{i+1}$ are complementing the interval $\br*{w_i, z_{i+1}}$.
    Hence, applying \autoref{ia:moebius-completition} on the last equation gives
    \begin{equation*}
      \prod_{i=1}^k \mu_{Q_1}\pa*{z_i, w_i} \, \Gamma_Q\pa*{w_i, z_{i+1}}
      = \prod_{i=1}^k \Inv_{Q'_i}\mu\pa*{z_i, w_i} \, \pa*{-\mu_{Q'_{i+1}} \pa*{w_i, z_{i+1}}}.
    \end{equation*}
  For all $i = 2, \dots ,k$ we have that $x \le z_1 < z_i < w_k < y$, hence $z_i\notin Q'$.
  So we apply \autoref{ia:gamma-formula} again:
  \begin{equation*}
    \prod_{i=1}^k \mu_{Q_1}\pa*{z_i, w_i} \, \Gamma_Q\pa*{w_i, z_{i+1}}
      = - \Inv_{Q'_1}\mu\pa*{z_1, w_1} \mu_{Q'_2}\pa*{w_1, z_2} \prod_{i=2}^k  \Gamma_{Q'}\pa*{z_i, w_i} \, \mu_{Q'_{i+1}} \pa*{w_i, z_{i+1}},
  \end{equation*}
  and form here we get the identity
  \begin{multline}
    \label{eq:tic:prod-identity}
    \prod_{i=1}^k \mu_{Q_1}\pa*{z_i, w_i} \, \Gamma_Q\pa*{w_i, z_{i+1}} \\
      = - \Inv_{Q'}\mu\pa*{x, w_1} \pa*{\prod_{i=1}^{k-1} \mu_{Q'_{i+1}} \pa*{w_i, z_{i+1}} \, \Gamma_{Q'}\pa*{z_{i+1}, w_{i+1}}} \mu_{Q'_{k+1}} \pa*{w_k, z_{k+1}}.
  \end{multline}

    Next, we take a look on the admissible form of $m^{-1} \cdot t_y / t_x$ which is
    \begin{equation*}
      m^{-1} \cdot \frac{t_y}{t_x} = \frac{t_{z_2}}{t_{w_1}} \cdots \frac{t_{z_k}}{t_{w_{k-1}}}\cdot \frac{t_y}{t_{w_k}}.
    \end{equation*}
    From \autoref{tp:theta-inv-coeffs},
    \begin{equation*}
      \br*{\Inv_{Q'} \theta\pa*{x, y}}_{m^{-1} \cdot t_y / t_x}
        = \Inv_{Q'} \mu\pa*{x, w_1} \prod_{i=1}^k \mu_{Q'_{i+1}} \pa*{w_i, z_{i+1}} \,\Gamma_{Q'}\pa*{z_{i+1}, w_{i+1}},
    \end{equation*}
    where $w_{k+1} = y$.
    And, by using \eqref{eq:tic:prod-identity}, then the identities $\Inv_Q\mu\pa*{x, z_1} = \Gamma_{Q'}\pa*{w_{k+1}, z_{k+1}} = 1$, and then \eqref{eq:tic:m-coeff} to obtain
    \begin{align*}
      \br*{\Inv_{Q'} \theta\pa*{x, y}}_{m^{-1} \cdot t_y / t_x}
        &= - \pa*{\prod_{i=1}^k \mu_{Q_i}\pa*{z_i, w_i} \, \Gamma_Q\pa*{w_i, z_{i+1}}} \Gamma_{Q'}\pa*{z_{k+1}, w_{k+1}} \\
        &= - \Inv_Q\mu\pa*{x, z_1} \prod_{i=1}^k \mu_{Q_i}\pa*{z_i, w_i} \, \Gamma_Q\pa*{w_i, z_{i+1}} \\
        &= - \br*{\Inv_Q \theta\pa*{x, y}}_m,
    \end{align*}
    as needed.
  \end{case}

  \begin{case}[When $x = z_1$ and $y = w_k$]
    \label{tic:case-fully-close}
    Similarly to proving \eqref{eq:tic:prod-identity}, we use \autoref{ia:gamma-formula}, then \autoref{ia:moebius-completition} and then \autoref{ia:gamma-formula} to get the identity
    \begin{multline}
      \label{eq:tic:prod-identity-2}
      \pa*{\prod_{i=1}^{k-1} \mu_{Q_i}\pa*{z_i, w_i} \, \Gamma_Q\pa*{w_i, z_{i+1}}} \mu_{Q_k}\pa*{z_k, w_k} \\
        = - \Inv_{Q'}\mu\pa*{x, w_1} \prod_{i=1}^{k-1} \mu_{Q'_{i+1}} \pa*{w_i, z_{i+1}} \, \Gamma_{Q'}\pa*{z_{i+1}, w_{i+1}}.
    \end{multline}
    Next, we focus on the admissible form of $m^{-1} \cdot t_y / t_x$ is
    \begin{equation*}
      m^{-1} \cdot \frac{t_y}{t_x} = \frac{t_{z_2}}{t_{w_1}} \cdots \frac{t_{z_k}}{t_{w_{k-1}}}.
    \end{equation*}
   Using \autoref{tp:theta-inv-coeffs} gives
    \begin{equation*}
      \br*{\Inv_{Q'} \theta\pa*{x, y}}_{m^{-1} \cdot t_y / t_x}
        = \Inv_{Q'} \mu\pa*{x, w_1} \prod_{i=1}^{k-1} \mu_{Q'_{i+1}} \pa*{w_i, z_{i+1}} \Gamma_{Q'}\pa*{z_{i+1}, w_{i+1}}.
    \end{equation*}
    Then, we use \eqref{eq:tic:prod-identity-2}, then the identities $\Inv_Q\mu\pa*{x, z_1} = \Gamma_Q\pa*{w_k, z_{k+1}} = 1$, and then \eqref{eq:tic:m-coeff} to obtain
    \begin{align*}
      \br*{\Inv_{Q'} \theta\pa*{x, y}}_{m^{-1} \cdot t_y / t_x}
        &= - \pa*{\prod_{i=1}^{k-1} \mu_{Q_i}\pa*{z_i, w_i} \, \Gamma_Q\pa*{w_i, z_{i+1}}} \mu_{Q_k} \pa*{z_k, w_k} \\
        &= - \Inv_Q\mu\pa*{x, z_1} \prod_{i=1}^k \mu_{Q_i}\pa*{z_i, w_i} \, \Gamma_Q\pa*{w_i, z_{i+1}} \\
        &= - \br*{\Inv_Q \theta\pa*{x, y}}_m,
    \end{align*}
    as needed
  \end{case}
  \begin{case}[When $x \ne z_1$]
    In this case, $t_x$ must be in the denominator of $m^{-1} \cdot t_y / t_x$.
    Therefore, by replacing $m$ with $m^{-1} \cdot t_y / t_x$ and swapping $Q$ and $Q'$, we get one of the other cases. \qedhere
  \end{case}
\end{proof}

\begin{lem}
\label{tp:theta-inversion}
Let $x < y$ and let $Q, Q' \subseteq P$ complementing the interval ${\br*{x, y}_P}$.
Set $\underline{t}^{-1} = \pa*{t_z^{-1}}_{z \in P}$.
Then,
\begin{equation*}
\Inv_Q \theta\pa*{x, y ; \underline{t}} = - \frac{t_y}{t_x} \Inv_{Q'} \theta\pa*{x, y ; \underline{t}^{-1}}.
\end{equation*}
\end{lem}
\begin{proof}
Let $m$ be some monomial in $A_P$. Its suffice to show that
\begin{equation}
\label{eq:tp:ti:need-to-proof}
\br*{\Inv_Q \theta\pa*{x, y ; \underline{t}}}_m = - \br*{\Inv_{Q'} \theta\pa*{x, y ; \underline{t}}}_{m^{-1} \cdot t_y / t_x}.
\end{equation}
We divide the proof into four cases:

\begin{case}[When $m$ is $\br*{x, y}_Q$-admissible with admissible form $\frac{t_{w_1}}{t_{z_1}} \cdots \frac{t_{w_k}}{t_{z_k}}$ and there exists $i$ such that $w_i \notin Q'$]
\label{tp:ti:case1}
In this case both sides of \eqref{eq:tp:ti:need-to-proof} are zero.
Indeed, $w_i \ne x$ and $w_i \ne y$ because $x, y \in Q'$.
Hence, $m^{-1} \cdot t_y / t_x$ must contain $t_{w_i}$ in its denominator and consequently $m^{-1} \cdot t_y / t_x$ is not $\br*{x, y}_{Q'}$-admissible.
Thus, by \autoref{tp:theta-inv-coeffs}, the right side of \eqref{eq:tp:ti:need-to-proof} is zero.
For the left side of \eqref{eq:tp:ti:need-to-proof},
by \autoref{ia:gamma-formula} we have $\Gamma_Q\pa*{w_i, z_{i+1}} = 0$,
then by using \autoref{tp:theta-inv-coeffs} we infer that the left side is also zero.
\end{case}

\begin{case}[When $m$ is $\br*{x, y}_Q$-admissible with admissible form $\frac{t_{z_1}}{t_{w_1}} \cdots \frac{t_{z_k}}{t_{w_k}}$ and for all $i$ we have $w_i \in Q'$]
\label{tp:ti:case2}
In this case we have that
\begin{equation*}
m^{-1} \cdot \frac{t_y}{t_x} = \frac{t_{w_1}}{t_x} \cdot \frac{t_{w_2}}{t_{z_1}} \cdots \frac{t_{w_k}}{t_{z_{k-1}}} \cdot \frac{t_y}{t_{z_k}},
\end{equation*}
and also
\begin{equation*}
x \le x \le w_1 \le z_1 \le w_2 \le  \dots \le z_k \le y.
\end{equation*}
Hence, $m^{-1} \cdot t_y / t_x$ is $\br*{x, y}_{Q'}$-admissible, and \autoref{tp:theta-inversion-coeff} implies \eqref{eq:tp:ti:need-to-proof}.
\end{case}

\begin{case}[When $m$ is not $\br*{x, y}_Q$-admissible and $m^{-1} \cdot t_y / t_x$ is not $\br*{x, y}_{Q'}$-admissible]
Here by \autoref{tp:theta-inv-coeffs} both sides of \eqref{eq:tp:ti:need-to-proof} are zeros and this equation holds.
\end{case}

\begin{case}[When $m$ is not $\br*{x, y}_Q$-admissible and $m^{-1} \cdot t_y / t_x$ is $\br*{x, y}_{Q'}$-admissible]
This case is either \autoref{tp:ti:case1} or \autoref{tp:ti:case2}, after replacing $m$ with $m^{-1} \cdot t_y / t_x$ and swapping $Q$ and $Q'$ in \eqref{eq:tp:ti:need-to-proof}.
\qedhere
\end{case}
\end{proof}

\subsection{The natural numbers poset}
We take a look on the poset $\NN$ with the partial order "$\mid$" and on the incidence algebra $I_{A_\NN}\pa*{\NN}$.
In this incidence algebra we have the M\"obius function and the $\theta$ polynomial.
We note that the known M\"obius function from number theory is connected the M\"obius function of incidence algebra over $\NN$ by the following relation:
\begin{equation*}
\mu\pa*{d, n} = \begin{cases}
\mu\pa*{\frac{n}{d}}, &d\mid n, \\
0, & \text{otherwise}.
\end{cases}
\end{equation*}

In this subsection we will always set $t_n = u^n$ where $u$ is some variable or expression.
Therefore, in this section we abbreviate and write $u$ instead of $\pa*{u^n}_{n\in\NN}$ on the third argument of elements of $I_{A_\NN}\pa*{\NN}$, i.e.,
\begin{equation*}
\iae\pa*{d, n; u} = \iae\pa*{d, n; \pa*{u^n}_{n\in\NN}} \quad \text{for all }\iae \in I_{A_\NN}\pa*{\NN}.
\end{equation*}
In particular, the $\theta$ polynomial is
\begin{equation}
\label{eq:theta-function-u}
  \theta\pa*{d, n; u} = \sum_{d \mid e \mid n} \mu\pa*{\frac{n}{e}} u^{e - d}.
\end{equation}

\begin{lem}
\label{d-formula}
For any positive integer $n \ge 1$,
\begin{equation}
\label{eq:d-function}
\dfunc_n\pa*{u, v} = - \frac{v^{-n+1}}{\prod_{\substack{e \mid n \\ e \ne n}}\pa*{r_e - 1}} \sum_{1, n \in Q \subseteq \br*{1, n}_\NN} \pa*{-1}^{\# Q} \,\Inv_Q \theta\pa*{1, n ; u} \,{\prod_{\substack{e\mid n \\ e \notin Q}} r_e},
\end{equation}
where $r_e = u^{n - e} v^{- n / e + 1}$.
\end{lem}
\begin{proof}
We start with simplifying the right side of \eqref{eq:d-function} which we denote with $\tilde \dfunc_n\pa*{u,v}$.
We use \autoref{ia:chain-inverse} on the right side of \eqref{eq:d-function},
\begin{equation*}
\tilde \dfunc_n\pa*{u, v} = - \frac{v^{-n+1}}{\prod_{\substack{e \mid n \\ e \ne n}}\pa*{r_e - 1}} \sum_{1, n \in Q \subseteq \br*{1, n}_\NN} \pa*{-1}^{\# Q} \, \sum_{\pa*{d_0, \dots, d_k} \in \Ccl^\ast_Q\pa*{1, n}}\pa*{-1}^k \pa*{\prod_{i=0}^{k-1}\theta\pa*{d_i, d_{i+1}; u}} \prod_{\substack{e\mid n \\ e \notin Q}} r_e \; .
\end{equation*}
By changing the order of summation we get that
\begin{equation}
\label{eq:dr:first-d}
\tilde \dfunc_n\pa*{u, v} = - \frac{v^{-n+1}}{\prod_{\substack{e \mid n \\ e \ne n}}\pa*{r_e - 1}} \sum_{\pa*{d_0, \dots, d_k} \in \Ccl^\ast_\NN\pa*{1, n}} \pa*{\prod_{i=0}^{k-1}\theta\pa*{d_i, d_{i+1}; u}} \sum_{\substack{Q \subseteq \br*{1, n}_\NN \\ d_0,\dots,d_k \in Q}} \pa*{-1}^{\# Q + k} \; \prod_{\substack{e\mid n \\ e \notin Q}} r_e.
\end{equation}

We focus on the inner sum of \eqref{eq:dr:first-d}.
By changing the variable $Q$ to $Q \setminus \set*{d_0, \dots, d_k}$, we get
\begin{equation}
  \label{eq:dr:inner-sum-intermid}
  \sum_{\substack{Q \subseteq \br*{1, n}_\NN \\ d_0,\dots,d_k \in Q}} \pa*{-1}^{\# Q + k} \; \prod_{\substack{e\mid n \\ e \notin Q}} r_e = \sum_{\substack{Q \subseteq \br*{1, n}_\NN \\ d_0,\dots,d_k \notin Q}} \pa*{-1}^{\#Q + 1} \prod_{\substack{e \mid n \\ e \notin Q \sqcup \set*{d_0, \dots, d_k}}} r_e.
\end{equation}
For a finite sequence $\pa*{a_i}_{i \in F}$ we have the identity
\begin{equation*}
\prod_{i \in F} \pa*{a_i - 1} = \sum_{F' \subseteq F} \pa*{-1}^{\# F'} \prod_{i \in F \setminus F'} a_i.
\end{equation*}
Using this identity on the right side of \eqref{eq:dr:inner-sum-intermid} gives
\begin{equation}
\label{eq:dr:inner-sum}
\sum_{\substack{Q \subseteq \br*{1, n}_\NN \\ d_0,\dots,d_k \in Q}} \pa*{-1}^{\# Q + k} \; \prod_{\substack{e\mid n \\ e \notin Q}} r_e
  = - \prod_{\substack{e \mid n \\ e \notin \set*{d_0, \dots, d_k}}} \pa*{r_e - 1}.
\end{equation}

We put \eqref{eq:dr:inner-sum} in \eqref{eq:dr:first-d} to obtain
\begin{equation*}
  \tilde \dfunc_n\pa*{u, v}
  = \frac{v^{-n+1}}{\prod_{\substack{e \mid n \\ e \ne n}}\pa*{r_e - 1}} \sum_{\pa*{d_0, \dots, d_k} \in \Ccl^\ast_\NN\pa*{1, n}} \pa*{\prod_{i=0}^{k-1}\theta\pa*{d_i, d_{i+1}; u}} \prod_{\substack{e \mid n \\ e \notin \set*{d_0, \dots, d_k}}} \pa*{r_e - 1}.
\end{equation*}
We insert the denominator inside the sum and get
\begin{equation}
\label{eq:dr:d-function-simple-form-with-r}
  \tilde \dfunc_n\pa*{u, v}
  = v^{-n+1} \sum_{\pa*{d_0, \dots, d_k} \in \Ccl^\ast_\NN\pa*{1, n}} \; \prod_{i=0}^{k-1} \frac{\theta\pa*{d_i, d_{i+1}; u}}{r_{d_i} - 1}.
\end{equation}
Plugging
\begin{equation*}
  v^{-n + 1} = \prod_{i=0}^{k-1} \frac{v^{n/d_{i+1}}}{v^{n/d_i}},
\end{equation*}
into \eqref{eq:dr:d-function-simple-form-with-r} gives
\begin{equation}
\label{eq:dr:d-function-simple-form}
\tilde \dfunc_n\pa*{u, v}
  = \sum_{\pa*{d_0, \dots, d_k} \in \Ccl^\ast_\NN\pa*{1, n}} \; \prod_{i=0}^{k-1} \frac{v^{n/d_{i+1}} \, \theta\pa*{d_i, d_{i+1}; u}}{u^{n - d_i}v  - v^{n / d_i}}.
\end{equation}

Next, we prove that $\dfunc_n\pa*{u, v} = \tilde\dfunc_n\pa*{u, v}$ with induction on $n$.
For $n=1$, the equality is clear.
For $n > 1$, the induction assumption and \eqref{eq:dr:d-function-simple-form} gives that for any $1 \ne d \mid n$,
\begin{equation*}
\dfunc_{n/d}\pa*{u^d, v}
  = \sum_{\pa*{d_0, \dots, d_k} \in \Ccl^\ast_\NN\pa*{1, n/d}} \; \prod_{i=0}^{k-1} \frac{v^{n/dd_{i + 1}} \, \theta\pa*{d_i, d_{i+1}; u^d}}{u^{n - dd_i}v  - v^{n / dd_i}}.
\end{equation*}
Using \eqref{eq:theta-function-u} we have $\theta\pa*{d_i, d_{i+1}; u^d} = \theta\pa*{dd_i, dd_{i+1}; u}$.
Hence,
\begin{equation}
\label{eq:dr:d-function-rec-1}
\dfunc_{n/d}\pa*{u^d, v}
  = \sum_{\pa*{d_0, \dots, d_k} \in \Ccl^\ast_\NN\pa*{1, n/d}} \; \prod_{i=0}^{k-1} \frac{v^{n/dd_{i+1}} \, \theta\pa*{dd_i, dd_{i+1}; u}}{u^{n - dd_i}v  - v^{n / dd_i}}.
\end{equation}

We take a look on the map
\begin{equation*}
\pa*{d_0,\dots, d_k} \mapsto \pa*{d d_0,\dots, d d_k}.
\end{equation*}
This map is a bijection from $\Ccl^\ast_\NN\pa*{1, n / d}$ to $\Ccl^\ast_\NN\pa*{d, n}$.
Therefore, we can change the variables of the sum in \eqref{eq:dr:d-function-rec-1} to get
\begin{equation}
\label{eq:dr:d-function-rec-2}
\dfunc_{n/d}\pa*{u^d, v}
  = \sum_{\pa*{d_0, \dots, d_k} \in \Ccl^\ast_\NN\pa*{d, n}} \; \prod_{i=0}^{k-1} \frac{v^{n/d_{i+1}} \, \theta\pa*{d_i, d_{i+1}; u}}{u^{n - d_i}v  - v^{n / d_i}}.
\end{equation}

Finally, we plug \eqref{eq:theta-function-u} and \eqref{eq:dr:d-function-rec-2} into \eqref{eq:d-function-rec}:
\begin{align*}
\dfunc_n\pa*{u, v}
 &= \frac{1}{u^{n-1} - v^{n-1}} \sum_{1\ne d \mid n} v^{n/d - 1} \; \theta\pa*{1,d; u} \sum_{\pa*{d_0, \dots, d_k} \in \Ccl^\ast_\NN\pa*{d, n}} \; \prod_{i=0}^{k-1} \frac{v^{n/d_{i+1}} \, \theta\pa*{d_i, d_{i+1}; u}}{u^{n - d_i}v  - v^{n / d_i}} \\
 &= \sum_{1\ne d \mid n} \; \sum_{\pa*{d_0, \dots, d_k} \in \Ccl^\ast_\NN\pa*{d, n}} \frac{v^{n/d} \, \theta\pa*{1, d; u}}{u^{n-1}v - v^n} \; \prod_{i=0}^{k-1} \frac{v^{n/d_{i+1}}\,\theta\pa*{d_i, d_{i+1}; u}}{u^{n - d_i}v  - v^{n / d_i}}.
\end{align*}
The two sums in the last line are running over all the proper chains from $1$ to $n$, when $d$ is chosen to be the first element in the chain after $1$.
Hence,
\begin{equation*}
\dfunc_n\pa*{u, v}
 = \sum_{\pa*{d_0, \dots, d_k} \in \Ccl^\ast_\NN\pa*{1, n}} \; \prod_{i=0}^{k-1} \frac{v^{n/d_{i+1}}\,\theta\pa*{d_i, d_{i+1}; u}}{u^{n - d_i}v  - v^{n / d_i}}.
\end{equation*}
And using \eqref{eq:dr:d-function-simple-form} finishes the proof.
\end{proof}

\begin{proof}[Proof of \autoref{d-function-inverse}]
When $n=1$ this proposition is clear, so we assume that $n > 1$.
From \autoref{d-formula} we have that
\begin{equation}
\label{eq:dfi:d-inverse}
\dfunc_n\pa*{u^{-1}, v^{-1}}
  =  - \frac{v^{n-1}}{\prod_{\substack{e \mid n \\ e \ne n}}\pa*{r_e^{-1} - 1}} \sum_{1, n \in Q \subseteq \br*{1, n}_\NN} \pa*{-1}^{\# Q} \,\Inv_Q \theta\pa*{1, n ; u^{-1}} \,{\prod_{\substack{e\mid n \\ e \notin Q}} r_e^{-1}},
\end{equation}
where $r_e = u^{n-e} v^{-n / e + 1}$.
First we focus on the fraction on the right side \eqref{eq:dfi:d-inverse}:
\begin{equation*}
\frac{v^{n-1}}{\prod_{\substack{e \mid n \\ e \ne n}}\pa*{r_e^{-1} - 1}}
  = \frac{v^{n - 1} r_n}{\prod_{\substack{e \mid n \\ e \ne n}}\pa*{r_e^{-1} - 1}}
  = \frac{v^{n - 1} \prod_{e \mid n} r_e}{\prod_{\substack{e \mid n \\ e \ne n}} \pa*{1 - r_e}}.
\end{equation*}
Thus,
\begin{equation*}
\frac{v^{n-1}}{\prod_{\substack{e \mid n \\ e \ne n}}\pa*{r_e^{-1} - 1}}
= \pa*{-1}^{\#\br*{1, n}_\NN + 1} \, v^{n - 1} \, \frac{\prod_{e \mid n} r_e}{\prod_{\substack{e \mid n \\ e \ne n}} \pa*{r_e - 1}}.
\end{equation*}
We put the last equation in \eqref{eq:dfi:d-inverse},
\begin{equation}
  \label{eq:dfi:d-inverse-2}
  \dfunc_n\pa*{u^{-1}, v^{-1}}
  =  \frac{v^{n-1}}{\prod_{\substack{e \mid n \\ e \ne n}}\pa*{r_e - 1}} \sum_{1, n \in Q \subseteq \br*{1, n}_\NN} \pa*{-1}^{\#\br*{1, n}_\NN + \# Q} \,\Inv_Q \theta\pa*{1, n ; u^{-1}} \,{\prod_{e \in Q} r_e}.
\end{equation}

Next, we focus on the sum of \eqref{eq:dfi:d-inverse-2}.
For any $1, n \in Q \subseteq \br*{1, n}_\NN$, set $\tilde Q = \pa*{\br*{1, n}_\NN \setminus Q} \cup \set*{1, n}$.
Since $Q = \pa*{\br*{1, n}_\NN \setminus \tilde Q} \cup \set*{1, n}$,
\begin{equation*}
{\prod_{e \in Q} r_e} = r_1 r_n {\prod_{\substack{e \mid n \\ e \notin \tilde Q}} r_e},
\end{equation*}
and then
\begin{equation}
  \label{eq:dfi:r-prod-comp}
  {\prod_{e \in Q} r_e} = u^{n-1} v^{-n + 1} {\prod_{\substack{e \mid n \\ e \notin \tilde Q}} r_e}.
\end{equation}
Moreover, we have that $\#\tilde Q = \#\br*{1,n}_\NN - \# Q + 2$.
Putting that and \eqref{eq:dfi:r-prod-comp} in \eqref{eq:dfi:d-inverse-2} gives
\begin{equation}
  \label{eq:dfi:d-inverse-3}
  \dfunc_n\pa*{u^{-1}, v^{-1}}
  =  \frac{u^{n+1}}{\prod_{\substack{e \mid n \\ e \ne n}}\pa*{r_e - 1}} \sum_{1, n \in Q \subseteq \br*{1, n}_\NN} \pa*{-1}^{\#\tilde Q} \,\Inv_Q \theta\pa*{1, n ; u^{-1}} \,\prod_{\substack{e \mid n \\ e \notin \tilde Q}} r_e.
\end{equation}

The posets $Q$ and $\tilde Q$ are complementing the interval $\br*{1, n}_\NN$, so applying \autoref{tp:theta-inversion} on \eqref{eq:dfi:d-inverse-3} gives
\begin{equation*}
  \dfunc_n\pa*{u^{-1}, v^{-1}}
  = -\frac{1}{\prod_{\substack{e \mid n \\ e \ne n}}\pa*{r_e - 1}} \sum_{1, n \in Q \subseteq \br*{1, n}_\NN} \pa*{-1}^{\#\tilde Q} \,\Inv_{\tilde Q} \theta\pa*{1, n ; u} \,\prod_{\substack{e \mid n \\ e \notin \tilde Q}} r_e.
\end{equation*}
We change the variable of the sum to obtain
\begin{equation}
  \label{eq:dfi:d-inverse-4}
  \dfunc_n\pa*{u^{-1}, v^{-1}}
  = -\frac{1}{\prod_{\substack{e \mid n \\ e \ne n}}\pa*{r_e - 1}} \sum_{1, n \in \tilde Q \subseteq \br*{1, n}_\NN} \pa*{-1}^{\#\tilde Q} \,\Inv_{\tilde Q} \theta\pa*{1, n ; u} \,\prod_{\substack{e \mid n \\ e \notin \tilde Q}} r_e.
\end{equation}

Finally, we use \autoref{d-formula} in \eqref{eq:dfi:d-inverse-4} and get
\begin{equation*}
\dfunc_n\pa*{u^{-1}, v^{-1}}
  = v^{n-1} \dfunc_n\pa*{u, v}. \qedhere
\end{equation*}
\end{proof}

% !TEX root = ./root.tex

\section{The roots of \texorpdfstring{$f_n$}{h\_n}}
Let $K$ be a $p$-adic field of degree $n$ over $\QQ_p$.
For an open set $U \subseteq K$ we define $\newroots{U}$ to be the number of roots of $f_n$ in the set $U$ which also generates $K$ i.e.\
\begin{equation*}
  \newroots{U} = \#\set*{x \in U : f_n\pa*{x} = 0 \text{ and } K = \QQ_p\br*{x}}.
\end{equation*}
If $U = K$ we abbreviate and write $\newroots* = \newroots{K}$.

In this section we calculate the expected value of $\newroots{U}$ for three sets $U = K, \Ocl_K, \maxideal$ in the case that $K / \QQ_p$ is unramified.
More precisely,
\begin{prop}
\label{hn-roots}
Let $K$ be an unramified field extension of $\QQ_p$ of degree $n$.
Then
\begin{align}
\label{eq:int-root-num}
\frac{1}{n} \Ex\br*{\newroots{\Ocl_K}} &= \frac{p^{n+1} - p^n}{p^{n+1} - 1} \cdot \dfunc*_n\pa*{p}, \\
\label{eq:nonumit-root-num}
\frac{1}{n} \Ex\br*{\newroots{\maxideal}} &= \frac{p - 1}{p^{n+1} - 1} \cdot \dfunc*_n \pa*{p^{-1}}, \qquad\text{and} \\
\label{eq:total-root-num}
\frac{1}{n} \Ex\br*{\newroots*} &= \frac{p - 1}{p^{n+1} - 1} \pa*{p^n \dfunc*_n\pa*{p} + \dfunc*_n \pa*{p^{-1}}}.
\end{align}
\end{prop}

Let $K / F$ be a finite extension of $p$-adic fields.
Define $\phi_{K/F}: \Ocl_K \to \RR^{+}$ to be the function
\begin{equation}
\label{eq:phi-function}
  \phi_{K/F}\pa*{x} = \lambda_K\pa*{\Ocl_F\br*{x}}.
\end{equation}
We prove the following properties on the function $\phi_{K/F}$.
\begin{lem}
\label{phi:disc-formula}
For all $x \in \Ocl_K$,
\begin{equation*}
  \phi_{K/F}\pa*{x}^2 = \pav[F]{\frac{\xdisc{x}}{\fdisc[K/F]}}.
\end{equation*}
\end{lem}
\begin{proof}
  If $K \ne F\br*{x}$ then $\xdisc{x} = 0$ and also $\lambda_K\pa*{\Ocl_F\br*{x}} \le \lambda_K\pa*{F\br*{x}} = 0$.
  So we assume that $K = F\br*{x}$.

  Let $\Bcl = \pa*{\omega_1, \dots, \omega_n}$ is an $\Ocl_F$-basis of $\Ocl_K$ and let $\Bcl_x = \pa*{1, x, \dots, x^n}$ be a $F$-basis of $K$.
  Note that $\Bcl$ is also a $F$-basis of $K$.
  Set $M$ to be the change of basis matrix from $\Bcl_x$ to $\Bcl$, so
  \begin{equation}
    \label{eq:phi:df:det-formula}
    \phi_{K/F}\pa*{x} = \pav[F]{\det M}.
  \end{equation}

  Moreover, from \eqref{eq:g:disc-conversion}
  \begin{equation*}
    \xdisc{1,x,\dots, x^n} = \pa*{\det M}^2 \xdisc{\omega_1, \dots, \omega_n}.
  \end{equation*}
  Since $\xdisc{\omega_1, \dots, \omega_n}$ is also the discriminant $\fdisc[K/F]$, we obtain
  \begin{equation*}
    \pa*{\det M}^2 = \frac{\xdisc{x}}{\fdisc[K/F]}.
  \end{equation*}
  Applying absolute value $\pav[F]{\cdot}$ on both sides and then using \eqref{eq:phi:df:det-formula} finish the proof.
\end{proof}

\begin{lem}
\label{phi:multiplication}
Assume $K/F$ is of degree $n$.
If $x \in \Ocl_K$ and $\alpha \in \Ocl_F$ then
\begin{equation*}
  \phi_{K/F}\pa*{\alpha x} = \abs*{\alpha}_F^{\binom{n}{2}} \phi_{K/F}\pa*{x}.
\end{equation*}
\end{lem}
\begin{proof}
Immediately from \autoref{phi:disc-formula} and the identity $\xdisc{\alpha x} = \alpha^{n\pa*{n-1}} \xdisc{x}$.
\end{proof}

We also need the following definition:
\begin{defn}
We say an integer element $\zeta \in \Ocl_K$ is \emph{inertial} if its degree over $\QQ_p$ equals to the degree of $\zeta \bmod \maxideal$ over $\FF_p$.
\end{defn}

Let $L$ be the unique subfield in $K$ which is unramified over $\QQ_p$ and of degree $m$.
For each element $\bar \zeta$ in the residue field of $K$ of degree $m$, we can lift $\bar\zeta$ to an element $\zeta$ in $L$.
It is clear that $\zeta$ is inertial.
Therefore, for each residue class of $\Ocl_K / \maxideal$ we can choose an inertial representative.

In addition, if $\zeta \in \Ocl_K$ is inertial and $q$ is the size of the residue field of $F$ then the degree of $\zeta$ over $F$ is the same as the degree of $\zeta \bmod \maxideal[F]$ over $\FF_q$.

\begin{lem}
\label{phi:faithful-nonunit}
Assume $K/F$ is an unramified extension.
Let $\zeta \in \Ocl_K$ be an inertial element of degree $m$ over $F$ and $x \in \maxideal[K]$.
Then
\begin{equation*}
  \phi_{K/F}\pa*{\zeta + x} = \phi_{K/F\br*{\zeta}}\pa*{x}.
\end{equation*}
\end{lem}
\begin{proof}
Set $y = \zeta + x$ and $L = F\br*{\zeta}$.
We have that $\pav[F]{\fdisc[K/F]} = \pav[L]{\fdisc[K/L]} = 1$.
So by \autoref{phi:disc-formula} it is sufficed to show that $\pav[F]{\xdisc{y}} = \pav[L]{\xdisc[K/L]{y}}$.

Let $f$ (resp. $g$) be the minimal polynomial of $y$ over $F$ (resp. $L$).
We have that
\begin{equation}
  \label{eq:phi-trans:disc-f}
  \pav[F]{\xdisc{y}}
    = \pav[F]{\anorm[K/F]{f'\pa*{y}}}
    = \pav[K]{f'\pa*{y}},
\end{equation}
and similarly
\begin{equation}
  \label{eq:phi-trans:disc-g}
  \pav[L]{\xdisc[K/L]{y}}
    = \pav[L]{\anorm[K/L]{g'\pa*{y}}}
    = \pav[K]{g'\pa*{y}}.
\end{equation}

Moreover, $g \mid f$ so there exists $h \in L\br*{X}$ such that $f = g h$.
Since $f' = g' h + g h'$ and $g\pa*{y} = 0$ we get that $f'\pa*{y} = g'\pa*{y} h\pa*{y}$.
Hence, form \eqref{eq:phi-trans:disc-f} and \eqref{eq:phi-trans:disc-g} we get that
\begin{equation}
  \label{eq:phi-trans:disc-split}
  \pav[F]{\xdisc{y}} = \pav[L]{\xdisc[K/L]{y}} \pav[K]{h\pa*{y}}.
\end{equation}

From the hypothesis $K/F$ is unramified, hence $K/F$ is Galois extension.
We denote $G = \Gal\pa*{K / F}$ and $H = \Gal\pa*{K / L} \le G$.
Since
\begin{equation*}
f\pa*{X} = \prod_{\sigma \in G} \pa*{X - \sigma\pa*{y}}, \quad \text{and} \quad
g\pa*{X} = \prod_{\sigma \in H} \pa*{X - \sigma\pa*{y}},
\end{equation*}
we have
\begin{equation*}
h\pa*{X} = \prod_{\substack{\sigma \in G \\ \sigma \notin H}} \pa*{X - \sigma\pa*{y}}.
\end{equation*}

Let $\pi$ be the common uniformizer of $F,L$ and $K$.
So
\begin{equation}
\label{eq:phi-trans:h-poly}
h\pa*{y} \equiv \prod_{\substack{\sigma \in G \\ \sigma \notin H}} \pa*{\zeta - \sigma\pa*{\zeta}} \pmod{\pi}.
\end{equation}
There is no $\sigma \in G \setminus H$ such that $\zeta \equiv \sigma\pa*{\zeta} \pmod{\pi}$.
Indeed, if $\zeta \equiv \sigma\pa*{\zeta} \pmod{\pi}$ then $\Ocl_L = \Ocl_F\br*{\zeta}$ implies that $\sigma|_L$ is in the inertia group of $L/F$.
But since $L/F$ is unramified, we get that $\sigma|_L = \id_L$ which contradicts $\sigma \notin H$.

Therefore, we get from \eqref{eq:phi-trans:h-poly} that $h\pa*{y} \not\equiv 0 \pmod{\pi}$ and consequently $\pav[K]{h\pa*{y}} = 1$.
Plugging this into \eqref{eq:phi-trans:disc-split} finish the proof.
\end{proof}

\begin{lem}
\label{phi-function-integral}
Let $K/F$ be an unramified extension of degree $n$ and let $r \in \CC$ with $\Im \pa*{r} > 0$.
Assume $q$ is the size of residue field of $F$.
Then
\begin{equation*}
\int_{\Ocl_K} \phi_{K/F}\pa*{x}^r \diff x = \dfunc_n\pa*{q, q^{-nr/2}}.
\end{equation*}
\end{lem}
\begin{proof}
We prove this with induction on $n$.
For $n = 1$ we have that $K = F$ and $\phi_{F/F}\pa*{x} = 1$.
So it is clear that $\int_{\Ocl_F} \phi_{F/F}\pa*{x}^r \diff x = \dfunc_1\pa*{q, q^{-nr/2}}$.

Assume $n > 1$.
Let $\pi$ be the common uniformizer of $F$ and $K$, and let $\Scl$ be a set of inertial representatives of the residue classes of $\Ocl_K /\maxideal$.

We have that $\Ocl_K = \bigsqcup_{\zeta \in \Scl} \pa*{\zeta + \maxideal[K]}$.
Hence,
\begin{equation*}
J := \int_{\Ocl_K} \phi_{K/F}\pa*{x}^r \diff x = \sum_{\zeta \in \Scl} \int_{\zeta + \maxideal[K]} \phi_{K/F}\pa*{x}^r \diff x
\end{equation*}
We apply change of variables $x \mapsto \zeta + \pi x$ to obtain
\begin{align*}
J
  &= \sum_{\zeta \in \Scl} \int_{\Ocl_K} \phi_{K/F}\pa*{\zeta + \pi x}^r \abs*{\pi}_K \diff x \\
  &= q^{-n} \sum_{\zeta \in \Scl} \int_{\Ocl_K} \phi_{K/F}\pa*{\zeta + \pi x}^r \diff x.
\end{align*}
Since $\zeta$ is inertial for each $\zeta \in \Scl$ we can use \autoref{phi:faithful-nonunit} to infer that
\begin{equation*}
J
  = q^{-n} \sum_{\zeta \in \Scl} \int_{\Ocl_K} \phi_{K/F\br*{\zeta}}\pa*{\pi x}^r \diff x.
\end{equation*}
Also, by \autoref{phi:multiplication} we get that
\begin{align*}
J
  &= q^{-n} \sum_{1\ne d \mid n} \; \sum_{\substack{\zeta \in \Scl \\ \deg \zeta = d}} \int_{\Ocl_K} \phi_{K/F\br*{\zeta}}\pa*{\pi x}^r \diff x \\
  &= q^{-n} \sum_{d \mid n} \; \sum_{\substack{\zeta \in \Scl \\ \deg_F \zeta = d}} \int_{\Ocl_K} \abs*{\pi}_{F\br*{\zeta}}^{\binom{n/d}{2}r} \phi_{K/F\br*{\zeta}}\pa*{x}^r \diff x \\
  &= q^{-n} \sum_{d \mid n} q^{-\binom{n /d}{2} d r} \sum_{\substack{\zeta \in \Scl \\ \deg_F \zeta = d}} \int_{\Ocl_K} \phi_{K/F\br*{\zeta}}\pa*{x}^r \diff x.
\end{align*}
Therefore, by the induction assumption we have that
\begin{equation*}
J
  = q^{-n} \sum_{1 \ne d \mid n} q^{-\binom{n /d}{2} d r} \sum_{\substack{\zeta \in \Scl \\ \deg_F \zeta = d}} \dfunc_{n/d}\pa*{q^d, q^{- nr/2}}
  + q^{-\binom{n}{2} r -n } \sum_{\substack{\zeta \in \Scl \\ \deg_F \zeta = 1}} J.
\end{equation*}
We set $\gen_d\pa*{q} = \#\set*{\zeta \in \Scl : d = \deg \zeta}$ and then
\begin{equation*}
J = q^{-n} \sum_{1 \ne d \mid n} q^{-\binom{n /d}{2} d r} \gen_d\pa*{q} \dfunc_{n/d}\pa*{q^d, q^{- nr/2}}
  + q^{-\binom{n}{2} r -n + 1} J.
\end{equation*}
By simple calculations we get that
\begin{equation}
\label{eq:pfi:main}
J = \frac{q^{- 1}}{q^{ n - 1} - q^{-\binom{n}{2}r}}\sum_{\substack{1\ne d \mid n}} q^{-\binom{n/d}{2}dr} \gen_d\pa*{q} \dfunc_{n/d}\pa*{q^d, q^{- nr/2}}.
\end{equation}

Next, we have that
\begin{equation}
\label{eq:pfi:num-ele-deg}
\gen_d\pa*{q}
  = q \sum_{e\mid d} \mu\pa*{\frac{d}{e}}q^{e - 1}.
\end{equation}
Indeed, since the elements of $\Scl$ are inertial
\begin{equation*}
\gen_d\pa*{q}
  =\#\set*{\zeta \in \Scl : d = \deg \zeta}
  = \#\set*{\bar \zeta \in \FF_{q^d} : d = \deg_{\FF_{q}} \bar \zeta}.
\end{equation*}
And since
\begin{equation*}
q^d = \sum_{e \mid d} \#\set*{\bar \zeta \in \FF_{q^d} : e = \deg_{\FF_q} \bar \zeta}
  = \sum_{e \mid d} \gen_e\pa*{q},
\end{equation*}
From M\"obius inversion formula we obtain
\begin{equation*}
\gen_d\pa*{q}
  = \sum_{e\mid d} \mu\pa*{\frac{d}{e}} q^e,
\end{equation*}
and \eqref{eq:pfi:num-ele-deg} follows immediately.

Finally, we plug in \eqref{eq:pfi:num-ele-deg} into \eqref{eq:pfi:main} gives
\begin{equation*}
J =
\frac{1}{q^{ n - 1} - q^{-\binom{n}{2}r}}\sum_{\substack{1\ne d \mid n}} q^{-\binom{n/d}{2}dr} \pa*{\sum_{e \mid d}\mu\pa*{\frac{d}{e}}q^{e-1}}\dfunc_{n/d}\pa*{q^d, q^{- nr/2}}.
\end{equation*}
We finish the proof by setting $u= q$ and $v=q^{-nr/2}$ in \eqref{eq:d-function-rec} and then comparing to the last equation.
\end{proof}

\begin{rem}
  \autoref{phi-function-integral} is related to Igusa's local zeta functions (see \cite{denef1991report} or \cite{igusa2007introduction}).
  This relation appear in the following manner:

  Let $\omega_1, \dots, \omega_n$ be a $\Ocl_F$-basis of $\Ocl_K$ we define the multivariate polynomial
  \begin{equation*}
    h\pa*{X_1, \dots, X_n} = \xdisc[K]{X_1 \omega_1 + \dots + X_n \omega_n}.
  \end{equation*}
  The polynomial $h$ is a homogeneous polynomial with coefficients in $F$.
  Set $s = 2r$, so the Igusa's zeta function of $h$ is
  \begin{equation*}
    Z_h\pa*{s} = \int_{\Ocl_F^n} \pav[F]{h\pa*{\underline{x}}}^s \diff \underline{x} = \int_{\Ocl_K} \phi_{K/F}\pa*{x}^r \diff x.
  \end{equation*}
  It is known that $Z_h$ is a rational function of $q^{-1}$ and $q^{-s}$,
  \autoref{phi-function-integral} gives this rational function explicitly.
  Moreover, by setting $Z\pa*{u, v} = \dfunc_n\pa*{u^{-1}, v^n}$ and using \autoref{d-function-inverse} we obtain $Z\pa*{u^{-1}, v^{-1}} = v^{\binom{n}{2}} Z\pa*{u, v}$.
  This functional equation has been conjectured for by Igusa \cite{igusa1989universal} under certain conditions on the polynomial $h$.
  One of the conditions is the existence of resolutions of singularities.
  Later, this conjecture was proved by Denef and Meuser \cite{denef1991functional}.
\end{rem}

\begin{proof}[Proof of \autoref{hn-roots}]
Let $U \subseteq \Ocl_K$ open set, so from \cite[Theorem 5.8]{caruso2021zeroes} we have
\begin{equation*}
  \Ex\br*{\newroots{U}} = \pav{\fdisc[K/\QQ_p]} \cdot \frac{p^{n+1} - p^n}{p^{n+1} - 1} \int_{U} \phi_{K / \QQ_p} \pa*{x} \diff x.
\end{equation*}
Since $K / \QQ_p$ is unramified,
\begin{equation}
  \label{eq:hr:ex-roots}
  \Ex\br*{\newroots{U}} = \frac{p^{n+1} - p^n}{p^{n+1} - 1} \int_{U} \phi_{K / \QQ_p} \pa*{x} \diff x.
\end{equation}

Therefore, \eqref{eq:int-root-num} follows immediately from \eqref{eq:d-star-function}, \eqref{eq:hr:ex-roots} and \autoref{phi-function-integral}.
For \eqref{eq:nonumit-root-num} we set $U = \maxideal$ in  \eqref{eq:hr:ex-roots} and get
\begin{equation}
  \label{eq:hr:ex-roots-ideal}
  \Ex\br*{\newroots{\maxideal}} = \frac{p^{n+1} - p^n}{p^{n+1} - 1} \int_{\maxideal} \phi_{K / \QQ_p} \pa*{x} \diff x.
\end{equation}
By a change of variable we get that
\begin{align*}
  \int_{\maxideal} \phi_{K / \QQ_p} \pa*{x}  \diff x
    &= \int_{\Ocl_K} \phi_{K / \QQ_p} \pa*{px} \cdot \pav[K]{p} \diff x \\
    &= p^{-n} \int_{\Ocl_k}\phi_{K / \QQ_p} \pa*{px} \diff x.
\end{align*}
Next we apply \autoref{phi:multiplication}, then \autoref{phi-function-integral} and then \autoref{d-function-inverse} we obtain
\begin{align*}
  \int_{\maxideal} \phi_{K / \QQ_p} \pa*{x} \diff x
    &= p^{-\binom{n}{2} - n} \int_{\Ocl_K} \phi_{K / \QQ_p} \pa*{x} \diff x \\
    &= p^{-\binom{n}{2} - n} \dfunc_n\pa*{p, p^{-n/2}} \\
    &= p^{-n} \dfunc_n\pa*{p^{-1}, p^{n/2}}.
\end{align*}
And we get \eqref{eq:nonumit-root-num} by plugging the last equation and \eqref{eq:d-star-function} into \eqref{eq:hr:ex-roots-ideal}.

Finally, we prove \eqref{eq:total-root-num}.
We have that
\begin{equation}
  \label{eq:hr:spliting-total-roots}
  \Ex\br*{\newroots*} = \Ex\br*{\newroots{\Ocl_K}} + \Ex\br*{\newroots{K\setminus \Ocl_K}}.
\end{equation}
If $x \in K\setminus \Ocl_k$ is a root of $f_n$, then $x^{-1}$ is a root of $X^n f_n\pa*{X^{-1}}$ and $x^{-1} \in \maxideal$.
Moreover, $X^n f_n\pa*{X^{-1}}$ has the same law as $f_n$.
Thus, $\Ex\br*{\newroots{K\setminus \Ocl_K}} = \Ex\br*{\newroots{\maxideal}}$.
Setting this in \eqref{eq:hr:spliting-total-roots} and then plugging \eqref{eq:int-root-num} and \eqref{eq:nonumit-root-num} finish the proof.
\end{proof}
% !TEX root = ./root.tex

\section{Main Proof}
\label{sec:main}
Set $\sigma = \pa*{n^1}$ as in \autoref{main-thm} and let $K / \QQ_p$ be an \'etale extension of splitting type $\sigma$.
Then $K$ must be a field of degree $n$ and unramified over $\QQ_p$.
Moreover, all the \'etale extensions of $\QQ_p$ with splitting type $\sigma$ are isomorphic to $K$.

We begin with proving the following relation between $\newroots*$ and the probabilities $\rho, \alpha$ and $\beta$. (c.f. \cite[Proposition 4.11]{caruso2021zeroes}).

\begin{lem}
  \label{probs-to-roots}
  Let $C$ be an event such that $\Pr\pa*{ \cdot \cond C}$ is well-defined.
  Then
  \begin{equation*}
    \Pr\pa*{E_\sigma \cond C} = \frac{1}{n} \Ex\br*{\newroots* \cond C}.
  \end{equation*}
\end{lem}
\begin{proof}
  First, we prove
  \begin{equation}
    \label{eq:ptr:caruso}
    \Pr\pa*{A_n \cong K \cond C} = \frac{\Ex\br*{\newroots* \cond C}}{\#\Aut_{\QQ_p}\pa*{K}}.
  \end{equation}
  This is similar to \cite[Proposition 5.11]{caruso2021zeroes}, and we follow its proof with small adjustments.

  We have that $\newroots* = \#\Hom_{\QQ_p}^{\text{surj}}\pa*{A_n, K}$.
  Also, since $A_n$ and $K$ are both of degree $n$ over $\QQ_p$ then any surjective homomorphism is an isomorphism.
  Therefore,
  \begin{equation*}
  \frac{\newroots*}{\#\Aut_{\QQ_p}\pa*{K}} = \inidicator_{A_n \cong K}.
  \end{equation*}
  By applying the expectation function $\Ex\br*{\ \cdot \cond C}$ on both sides of the equation and using the linearity of expectation we obtain \eqref{eq:ptr:caruso}.

  The field $K$ is the unique (up to an isomorphism) \'etale algebra extending $\QQ_p$ with splitting type $\sigma$.
  Therefore, the event $E_\sigma$ is equivalent to the event that $A_n \cong K$.
  Moreover, it is known that $\#\Aut_{\QQ_p}\pa*{K} = n$.
  Putting those facts in \eqref{eq:ptr:caruso} finish the proof.
\end{proof}

We now prove each of the equations in \autoref{main-thm}.

\begin{proof}[Proof of \eqref{eq:rho-formula}]
  This is an immediate consequence of \eqref{eq:total-root-num} and \autoref{probs-to-roots}.
\end{proof}

\begin{proof}[Proof of \eqref{eq:alpha-formula}]
  If $f_n$ is not a primitive polynomial we can divide it with $p$ and get a random polynomial with the same law as $f_n$.
  Therefore,
  \begin{equation}
  \label{eq:mr:primitive}
  \Ex\br*{\newroots{\Ocl_K}} = \Ex\br*{\newroots{\Ocl_K} \cond f_n \text{ primitive}}.
  \end{equation}

  Assume $f_n$ is primitive and $p \mid \rv_n$ occurs then $\bar f_n := f_n \bmod p$ is nonzero polynomial with degree $< n$.
  From Hensel's lemma (see \cite[Lemma II.4.6]{neukirch1999algebraic}) there exists $f, g \in \ZZ_p\br*{X}$ such that $\deg f = \deg \bar f_n < n$, $f \equiv f_n \pmod p$, $g \equiv 1 \pmod p$ and $f_n = fg$.

  Let $x \in \Ocl_K$ be a root of $f_n$, so $f\pa*{x} = 0$ otherwise $g \pa*{x} = 0$ which contradicts $g \equiv 1 \pmod p$.
  But since $\deg f < n$ we get that $K \ne \QQ_p\br*{x}$.
  Therefore, there are no roots of $f_n$  in $\Ocl_K$ which are generators of $K$ i.e. $\newroots{\Ocl_K} = 0$.

  From the assumption we have that $\Ex\br*{\newroots{\Ocl_K} \cond f_n \text{ primitive} \land p\mid \rv_n} = 0$.
  So, using the law of total expectation on the right side of \eqref{eq:mr:primitive} gives
  \begin{align*}
  \Ex\br*{\newroots{\Ocl_K}} &= \Pr\pa*{p\nmid \rv_n \cond f_n \text{ primitive}} \, \Ex\br*{\newroots{\Ocl_K} \cond p\nmid \rv_n} \\
    &= \frac{p^n\pa*{p - 1}}{p^{n+1} - 1} \cdot \Ex\br*{\newroots{\Ocl_K} \cond p\nmid \rv_n}.
  \end{align*}
  Dividing the variable $\rv_i$ with a unit does not change its law, so we can replace the condition in the expectation with $\rv_n = 1$ i.e.
  \begin{equation*}
  \Ex\br*{\newroots{\Ocl_K}} = \frac{p^n\pa*{p - 1}}{p^{n+1} - 1} \cdot \Ex\br*{\newroots{\Ocl_K} \cond f_n \text{ monic}}.
  \end{equation*}
  Also, under this condition all the roots of $f_n$ are in $\Ocl_K$ and thus $\newroots{\Ocl_K} = \newroots*$.
  So
  \begin{equation*}
  \Ex\br*{\newroots{\Ocl_K}} = \frac{p^{n+1} - p^n}{p^{n+1} - 1} \cdot \Ex\br*{\newroots* \cond f_n \text{ monic}},
  \end{equation*}
  Plugging \eqref{eq:int-root-num} into the last equation gives.
  \begin{equation*}
    \frac{1}{n} \Ex\br*{\newroots* \cond f_n \text{ monic}} = \dfunc*_n\pa*{p},
  \end{equation*}
  and the proof is finished by \autoref{probs-to-roots}.
\end{proof}

\begin{proof}[Proof of \eqref{eq:beta-formula}]
  Using similar arguments as in the proof of \eqref{eq:alpha-formula} we have that
  \begin{equation}
  \label{eq:mir:primitive}
  \Ex\br*{\newroots{\maxideal}} = \Ex\br*{\newroots{\maxideal} \cond f_n \text{ primitive}}.
  \end{equation}

  Assume $f_n$ is primitive and $f_n \not\equiv u X^n \pmod{p}$ for each unit $u \in \FF_p^\times$.
  So there exists $m < n$ and $\bar g \in \FF_p\br*{X}$ such that $f_n \equiv X^m \bar g \pmod p$ and $\bar g \pa*{0} \not\equiv 0 \pmod p$.
  By Hensel's lemma (see \cite[Lemma II.4.6]{neukirch1999algebraic}) there exists $f, g \in \ZZ_p\br*{X}$ such that $\deg f = m < n$, $f \equiv X^m \pmod p$, $g \equiv \bar g \pmod p$ and $f_n = f g$.

  Let $x \in \maxideal$ be a root of $f_n$, so $f\pa*{x} = 0$ otherwise $\bar g \pa*{0} \equiv g \pa*{x} = 0 \pmod p$ which is contradiction for the choice of $\bar g$.
  But since $\deg f < n$ we get that $K \ne \QQ_p\br*{x}$.
  Therefore, there are no roots of $f_n$ in $\maxideal$ which are generators of $K$ i.e. $\newroots{\maxideal} = 0$.

  From the assumption we have that $\Ex\br*{\newroots{\Ocl_K} \cond \bigwedge_{u\in \FF_p^\times} f_n \not\equiv u X^n \pmod p} = 0$.
  So, by applying the law of total expectation on the right of \eqref{eq:mir:primitive},
  \begin{multline*}
  \Ex\br*{\newroots{\maxideal}} \\
  \begin{aligned}
    &= \Pr\pa*{\bigvee_{u\in \FF_p^\times} f_n \equiv u X^n \pmod p \cond  f_n \text{ primitive}} \Ex\br*{\newroots{\maxideal} \cond \bigvee_{u\in \FF_p^\times} f_n \equiv u X^n \pmod p} \\
    &= \frac{p-1}{p^{n+1} - 1} \cdot \Ex\br*{\newroots{\maxideal} \cond \bigvee_{u\in \FF_p^\times} f_n \equiv u X^n \pmod p}.
  \end{aligned}
  \end{multline*}
  Dividing the variable $\rv_i$ with a unit does not change its law, so we can replace the condition in the expectation with the condition that $f_n$ monic and $f_n \equiv X^n \pmod p$ i.e.
  \begin{equation*}
  \Ex\br*{\newroots{\maxideal}} = \frac{p-1}{p^{n+1} - 1} \cdot \Ex\br*{\newroots{\maxideal} \cond f_n \text{ monic and } f_n \equiv X^n \pmod{p}}.
  \end{equation*}
  Also, under this condition all roots of $f_n$ are in $\maxideal$ and thus $\newroots{\maxideal} = \newroots*$.
  Thus,
  \begin{equation*}
  \Ex\br*{\newroots{\maxideal}} = \frac{p-1}{p^{n+1} - 1} \cdot \Ex\br*{\newroots* \cond f_n \text{ monic and } f_n \equiv X^n \pmod{p}}.
  \end{equation*}
  From \eqref{eq:nonumit-root-num} we get that
  \begin{equation*}
    \frac{1}{n} \Ex\br*{\newroots* \cond f_n \text{ monic and } f_n \equiv X^n \pmod{p}} = \dfunc*_n\pa*{p^{-1}},
  \end{equation*}
  and the proof is finished by \autoref{probs-to-roots}.
\end{proof}

\newpage
\bibliographystyle{alpha}
\bibliography{main}

\end{document}